\documentclass[10pt]{article}

\usepackage[hmargin=30mm,vmargin=40mm]{geometry}

\usepackage{amsmath}      
\usepackage{amssymb}      
\usepackage{amsthm}       
\usepackage{enumerate}    
\usepackage{graphicx}     
\usepackage{mathrsfs}     
\usepackage{url}          

\newtheorem{theorem}{Theorem}[section]
\newtheorem{corollary}[theorem]{Corollary}
\newtheorem{lemma}[theorem]{Lemma}
\newtheorem{algorithm}[theorem]{Algorithm}

\theoremstyle{definition}
\newtheorem{defn}[theorem]{Definition}
\newtheorem*{remark}{Remark}

\numberwithin{equation}{section}

\newcommand{\basis}{\beta}
\newcommand{\basiselt}{\beta}
\newcommand{\basismat}{\redm_\basis}
\newcommand{\basisinv}{\basismat^{-1}}
\newcommand{\bconst}{\delta(\redm)}

\newcommand{\cpp}{C\nolinebreak\hspace{-.05em}\raisebox{.4ex}{%
    \tiny\textbf{+}}\nolinebreak\hspace{-.10em}\raisebox{.4ex}{\tiny\textbf{+}}}
\newcommand{\polyn}{\varphi(n)}

\newcommand{\R}{\mathbb{R}}
\newcommand{\qproj}{\mathscr{Q}}
\newcommand{\qcone}{\qproj^\vee}
\newcommand{\redm}{\widetilde{M}}
\newcommand{\regina}{\emph{Regina}}
\newcommand{\snap}{\emph{Snap}}
\newcommand{\unk}{-}
\newcommand{\Z}{\mathbb{Z}}

\DeclareMathOperator{\adj}{adj}
\DeclareMathOperator{\lcm}{lcm}
\DeclareMathOperator{\rank}{rank}

\title{A tree traversal algorithm for decision problems \\
in knot theory and 3-manifold topology}
\author{Benjamin A.~Burton and Melih Ozlen}

\date{March 31, 2012}

\begin{document}

\maketitle

\begin{abstract}
    In low-dimensional topology, many important decision algorithms
    are based on normal surface enumeration, which is a form of
    vertex enumeration over a high-dimensional and highly
    degenerate polytope.
    Because this enumeration is subject to extra combinatorial
    constraints, the only practical algorithms to date
    have been variants of the classical double description method.
    In this paper we present the first practical normal surface enumeration
    algorithm that breaks out of the double description paradigm.
    This new algorithm is based on a tree traversal with feasibility
    and domination tests,
    and it enjoys a number of advantages over the double description method:
    incremental output, significantly lower time and space complexity,
    and a natural suitability for parallelisation.
    Experimental comparisons of running times are included.


\end{abstract}

%
%

\section{Introduction}

Much research in low-dimensional topology has been driven by
decision problems.  Examples from knot theory include
\emph{unknot recognition} (given a polygonal representation of a knot,
decide whether it is equivalent to a trivial unknotted loop), and the more
general problem of
\emph{knot equivalence} (given two polygonal representations of knots,
decide whether they are equivalent to each other).
Analogous examples from 3-manifold
topology include \emph{3-sphere recognition} (given a triangulation of a
3-manifold, decide whether this 3-manifold is the 3-dimensional sphere),
and the more general \emph{homeomorphism problem} (given two triangulations
of 3-manifolds, decide whether these 3-manifolds are homeomorphic, i.e.,
``topologically equivalent'').

Significant progress has been made on these problems over the past 50 years.
For instance, Haken introduced an algorithm for recognising the unknot
in the 1960s \cite{haken61-knot}, and Rubinstein presented the first 3-sphere
recognition algorithm in the early 1990s \cite{rubinstein95-3sphere}.
Since Perelman's recent proof of the geometrisation conjecture
\cite{kleiner08-perelman}, we now have a complete (but extremely complex)
algorithm for the homeomorphism problem,
pieced together though a diverse array of techniques by many different
authors \cite{jaco05-lectures-homeomorphism,matveev03-algms}.

A key problem remains with many 3-dimensional decision algorithms
(including all of those mentioned above): the best known algorithms run
in exponential time (and for some problems, worse), where the input size is
measured by the number of crossings in a knot diagram, or the number of
tetrahedra in a 3-manifold triangulation.  This severely limits the
practicality of such algorithms.  For instance, it took around 30~years
to resolve Thurston's well-known conjecture regarding the Weber-Seifert
space \cite{birman80-problems}---although an algorithmic solution was
known in the 1980s \cite{jaco84-haken}, it took another 25~years of
development before the necessary computations
could be performed \cite{burton12-ws}.  Here the input size was $n=23$.

The most successful machinery for 3-dimensional decision problems has
been \emph{normal surface theory}.  Normal surfaces were introduced by
Kneser \cite{kneser29-normal} and later developed by Haken for algorithmic use
\cite{haken61-knot,haken62-homeomorphism}, and they play a key role
in all of the decision algorithms mentioned above.
In essence, they allow us to convert difficult ``topological searches''
into a ``combinatorial search'' in which we enumerate vertices
of a high-dimensional polytope called the \emph{projective solution space}.%
\footnote{For some topological algorithms, vertex enumeration is not
enough: instead we must enumerate a Hilbert basis for a polyhedral cone.}
In Section~\ref{s-prelim} we outline the combinatorial aspects of
the projective solution space;
see \cite{hass99-knotnp} for a broader introduction to normal surface
theory in a topological setting.

For many topological decision problems, the computational bottleneck is
precisely the vertex enumeration described above.  Vertex enumeration over
a polytope is a well-studied problem, and several families of algorithms appear
in the literature.  These include
\emph{backtracking algorithms} \cite{balinski61-backtrack,fukuda97-analysis},
pivot-based \emph{reverse search algorithms}
\cite{avis00-revised,avis92-pivot,dyer83-complexity},
and the inductive \emph{double description method} \cite{motzkin53-dd}.
All have worst-case running times that are
super-polynomial in both the input \emph{and} the output size, and it is not
yet known whether a polynomial time algorithm (with respect to the combined
input and output size) exists.%
\footnote{Efficient algorithms do exist for certain classes of polytopes.
For example, in the case of non-degenerate polytopes, the reverse search
algorithm of Avis and Fukuda has virtually no space requirements beyond
storing the input, and has a running time polynomial in the combined input
and output size \cite{avis92-pivot}.}
For topological problems, the corresponding running times are
exponential in the ``topological input size''; that is, the number of
crossings or tetrahedra.

Normal surface theory poses particular challenges
for vertex enumeration:
\begin{itemize}
    \item The projective solution space is a highly degenerate polytope
    (that is, a $d$-dimensional polytope in which vertices typically
    belong to significantly more than $d$ facets).
    \item Exact arithmetic is required.
    This is because each vertex of the projective
    solution space is a rational vector that effectively ``encodes'' a
    topological surface, and topological decision algorithms require us
    to scale each vertex to its smallest integer multiple in order to
    extract the relevant topological information.
    \item We are not interested in \emph{all} of the vertices of the
    projective solution space, but just those that satisfy a powerful
    set of conditions called the \emph{quadrilateral constraints}.
    These are extra combinatorial constraints that eliminate all but a small
    fraction of the polytope,
    and an effective enumeration algorithm should be able to enforce
    them as it goes, not simply use them as a post-processing output filter.
\end{itemize}

All well-known implementations of normal surface enumeration
\cite{burton04-regina,fxrays}
are based on the double description method: degeneracy and exact
arithmetic do not pose any difficulties, and it is simple to embed the
quadrilateral constraints directly into the algorithm \cite{burton10-dd}.
Section~\ref{s-prelim} includes a discussion of why backtracking and
reverse search algorithms have not been used to date.
Nevertheless, the double description method does suffer from
significant drawbacks:
\begin{itemize}
    \item The double description method can suffer from a
    \emph{combinatorial explosion}: even if both the input and output
    sizes are small, it can create
    extremely complex intermediate polytopes where the number of
    vertices is super-polynomial in both the input and output sizes.
    This in turn leads to an unwelcome explosion in both time and memory
    use, a particularly frustrating situation for normal surface
    enumeration where the output size is often small (though still
    exponential in general) \cite{burton10-complexity}.
    \item It is difficult to parallelise the double description method,
    due to its inductive nature.  There have been recent efforts in this
    direction \cite{terzer10-parallel},
    though they rely on a shared memory model, and for some
    polytopes the speed-up factor plateaus when the number of processors
    grows large.
    \item Because of its inductive nature, the double description method
    typically does not output \emph{any} vertices until the entire
    algorithm is complete.  This effectively removes the possibility
    of early termination (which is desirable for many topological
    algorithms), and it further impedes any attempts at parallelisation.
\end{itemize}

In this paper we present a new algorithm for normal surface enumeration.
This algorithm belongs to the backtracking family, and is described in
detail in Section~\ref{s-algm}.  Globally, the algorithm is structured
as a tree traversal, where the underlying tree is a decision tree
indicating which coordinates are zero and which are non-zero.  Locally,
we move through the tree by performing incremental feasibility tests
using the dual simplex method (though interior point methods could equally
well be used).

Fukuda et~al.\ \cite{fukuda97-analysis} discuss an impediment to
backtracking algorithms, which is the need to solve the NP-complete
\emph{restricted vertex problem}.  We circumvent this by
ordering the tree traversal in a careful way and introducing
\emph{domination tests} over the set of previously-found vertices.
This introduces super-polynomial time and space trade-offs, but for
normal surface enumeration these trade-offs are typically not
severe (we quantify this both theoretically and empirically within
Sections~\ref{s-complexity} and~\ref{s-perf} respectively).

Our algorithm is well-suited to normal surface enumeration.
It integrates the quadrilateral constraints seamlessly into the tree
structure.  Moreover, we are often able to perform exact arithmetic
using native machine integer types (instead of arbitrary precision
integers, which are significantly slower).  We do this by exploiting the
sparseness of the equations that define the projective solution space,
and thereby bounding the integers that appear in intermediate
calculations.  The details appear in Section~\ref{s-bounds}.

Most significantly, this new algorithm is the first normal surface
enumeration algorithm to break out of the double description paradigm.
Furthermore, it enjoys a number of advantages over previous algorithms:
\begin{itemize}
    \item The theoretical worst-case bounds on both time and space
    complexity are significantly better for the new algorithm.
    In particular, the space complexity is a small polynomial in
    the combined input and output size.  See Section~\ref{s-complexity}
    for details.
    \item The new algorithm lends itself well to parallelisation,
    including both shared memory models and distributed processing,
    with minimal need for inter-communication and synchronisation.
    \item The new algorithm produces incremental output, which makes it
    well-suited for early termination and further parallelisation.
    \item The tree traversal supports a natural sense of ``progress
    tracking'', so that users can gain a rough sense for how far they are
    through the enumeration procedure.
\end{itemize}

In Section~\ref{s-perf} we measure the performance of the new algorithm
against the prior state of the art, using a rich test suite with
hundreds of problems that span a wide range of difficulties.
For simple problems we find the new algorithm slower, but as the
problems become more difficult the new algorithm quickly outperforms the
old, often running orders of magnitude faster.
We conclude that for difficult problems this new tree traversal algorithm
is superior, owing to both its stronger practical performance and its
desirable theoretical attributes as outlined above.

\section{Preliminaries} \label{s-prelim}

For decision problems based on normal surface theory, the input is typically a
\emph{3-manifold triangulation}.  This is a collection of $n$ tetrahedra
where some or all of the $4n$ faces are affinely identified (or ``glued
together'') in pairs, so that the resulting topological space is a
3-manifold (possibly with boundary).
The \emph{size of the input} is measured by the number
of tetrahedra, and we refer to this quantity as $n$ throughout this paper.

We allow two faces of the same tetrahedron to be identified together;
likewise, we allow different edges or vertices of the same tetrahedron
to be identified.  Some authors refer to such triangulations as
\emph{generalised triangulations}.
In essence, we allow the tetrahedra to be ``bent'' or
``twisted'', instead of insisting that they be rigidly embedded in some
larger space.  This more flexible definition allows us to keep the input
size $n$ small, which is important when dealing with exponential algorithms.

\begin{figure}[htb]
    \centering
    \includegraphics[scale=0.5]{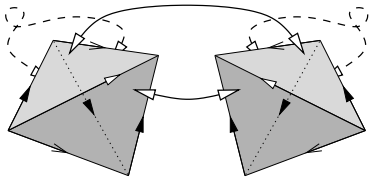}
    \caption{A 3-manifold triangulation of size $n=2$}
    \label{fig-s2xs1}
\end{figure}

Figure~\ref{fig-s2xs1} illustrates a triangulation
with $n=2$ tetrahedra: the back two faces of each
tetrahedron are identified with a twist, and the front two faces of the
first tetrahedron are identified with the front two faces of the second
tetrahedron.  As a consequence, all eight tetrahedron vertices become
identified together, and the edges become identified
into three equivalence classes as indicated by the three different arrowheads.
We say that the overall triangulation has \emph{one vertex} and
\emph{three edges}.
The underlying 3-manifold represented by this triangulation is the
product space $S^2 \times S^1$.

A \emph{closed} triangulation is one in which
all $4n$ tetrahedron faces are identified in $2n$ pairs (as in the
example above).
A \emph{bounded} triangulation is one in which some of the $4n$ tetrahedron
faces are left unidentified; together these unidentified faces form the
\emph{boundary} of the triangulation.  The underlying topological spaces
for closed and bounded triangulations are closed and bounded 3-manifolds
respectively.

There are of course other ways of representing a 3-manifold (for
instance, Heegaard splittings or Dehn surgeries).  We use triangulations
because they are ``universal'': it is typically easy to convert some
other representation into a triangulation, whereas it can sometimes be
difficult to move in the other direction.  In particular, when we are
solving problems in knot theory, we typically work with a triangulation
of the \emph{knot complement}---that is, the 3-dimensional space
surrounding the knot.

In normal surface theory, interesting surfaces within a 3-manifold
triangulation can be encoded as integer points in $\R^{3n}$.
Each such point $\mathbf{x} \in \R^{3n}$ satisfies the following
conditions:
\begin{itemize}
    \item $\mathbf{x} \geq \mathbf{0}$ and $M \mathbf{x} = \mathbf{0}$,
    where $M$ is a matrix of \emph{matching equations} that depends on
    the input triangulation;
    \item $\mathbf{x}$ satisfies the \emph{quadrilateral constraints}:
    for each triple of coordinates $(x_1,x_2,x_3)$, $(x_4,x_5,x_6)$, \ldots,
    $(x_{3n-2},x_{3n-1},x_{3n})$, at most one of the three entries
    in the triple can be non-zero.  There are $n$ such triples, giving
    $n$ such restrictions in total.
\end{itemize}
Any point $\mathbf{x} \in \R^{3n}$ that satisfies all of these
conditions (that is, $\mathbf{x} \geq \mathbf{0}$, $M \mathbf{x} = \mathbf{0}$
and the quadrilateral constraints) is called \emph{admissible}.

It should be noted that normal surface theory is often described using a
different coordinate system ($\R^{7n}$, also known as \emph{standard
coordinates}).
Our coordinates in $\R^{3n}$ are known as \emph{quadrilateral coordinates},
and were introduced by Tollefson \cite{tollefson98-quadspace}.
They are ideal for computation
because they are smaller and faster, and because information that
is ``lost'' from standard coordinates is typically quick to reconstruct
\cite{burton09-convert}.

These coordinates have a natural geometric interpretation, which we
outline briefly.
Typically one can arrange for an interesting surface within the 3-manifold to
intersect each tetrahedron of the triangulation in a collection of
disjoint triangles and\,/\,or quadrilaterals, as illustrated in
Figure~\ref{fig-normaldiscs}.  The corresponding admissible integer point in
$\R^{3n}$ counts how many quadrilaterals slice through each tetrahedron
in each of three possible directions.  From this information, the
locations of the triangles can also be reconstructed, under some weak
assumptions about the surface.

\begin{figure}[htb]
    \centering
    \includegraphics[scale=0.5]{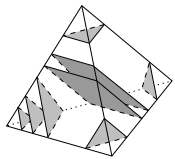}
    \caption{Triangles and quadrilaterals in a tetrahedron}
    \label{fig-normaldiscs}
\end{figure}

The matching equations ensure that triangles and quadrilaterals in
adjacent tetrahedra can be joined together, and the quadrilateral constraints
ensure that the resulting surface does not intersect itself.
We do not use this geometric interpretation in this paper,
and we do not discuss it further; for details the
reader is referred to \cite{hass99-knotnp}.

We recount some well-known facts about the matching equations
\cite{burton09-convert,tillmann08-finite,tollefson98-quadspace}:
\begin{lemma}
    Let $e$ and $v$ be the number of edges and vertices of the
    triangulation that do not lie within the boundary.
    Then there are $e$ matching equations; that is, $M$ is an
    $e \times 3n$ matrix.  In the case of a closed triangulation,
    we have $e=n+v > n$.

    In general one or more matching equations may be redundant.
    In the case of a closed triangulation, the rank of $M$ is precisely $n$.
\end{lemma}

\begin{lemma} \label{l-matching-sparse}
    The matrix $M$ is sparse with small integer entries.
    Each column contains at most four non-zero entries, each of which is
    $\pm1$, $\pm2$, $\pm3$ or $\pm4$.  Moreover, the sum of absolute values
    in each column is at most $4$; that is, $\sum_i |M_{i,c}| \leq 4$
    for all $c$.  If the triangulation is orientable, the only possible
    non-zero entries are $\pm1$ and $\pm2$.
\end{lemma}

We define the \emph{solution cone} $\qcone$ to be the set
\[ \qcone = \left\{\left. \mathbf{x} \in \R^{3n}\,\right|\,
    \mathbf{x} \geq \mathbf{0}\ \mathrm{and}
    \ M\mathbf{x} = \mathbf{0}\right\}.
\]
This is a polyhedral cone with apex at the origin.
From above, it follows that interesting surfaces within
our 3-manifold are encoded as integer points within the solution cone.
The \emph{projective solution space} $\qproj$ is defined to be a
cross-section of this cone:
\[ \qproj = \left\{ \mathbf{x} \in \qcone\,\left|
    \,\sum x_i = 1\right.\right\}.
\]
It follows that the projective solution space is a (bounded) polytope
with rational vertices.
In general, both the solution cone and the projective solution space
are highly degenerate (that is, extreme rays of $\qcone$
and vertices of $\qproj$ often belong to many more facets than required
by the dimension).

For polytopes and polyhedra, we follow the terminology of Ziegler
\cite{ziegler95}: polytopes are always bounded, polyhedra may be
either bounded or unbounded, and all polytopes and polyhedra are
assumed to be convex.

It is important to note that the quadrilateral constraints do not
feature in the definitions of $\qcone$ or $\qproj$.  These constraints
are combinatorial in nature, and they cause significant theoretical
problems: they are neither linear nor convex, and the region of
$\qproj$ that obeys them can be disconnected or can have non-trivial
topology (such as ``holes'').
Nevertheless, they are
extremely powerful: although the vertices of $\qproj$ can be numerous,
typically just a tiny fraction of these vertices satisfy the
quadrilateral constraints
\cite{burton10-complexity,burton11-asymptotic}.

For many topological decision problems that use normal surface theory, a
typical algorithm runs as follows:
\begin{enumerate}
    \item Enumerate all admissible vertices of the projective solution space
    $\qproj$ (that is, all vertices of $\qproj$ that satisfy the quadrilateral
    constraints);
    \item Scale each vertex to its smallest integer multiple,
    ``decode'' it to obtain a surface within the underlying 3-manifold
    triangulation, and test whether this surface is ``interesting'';
    \item Terminate the algorithm once an interesting surface is found,
    or once all vertices are exhausted.
\end{enumerate}
The definition of ``interesting'' varies for different decision problems.
For more of the overall topological context see \cite{hass99-knotnp},
and for more details on quadrilateral coordinates and the projective solution
space see \cite{burton09-convert}.

As indicated in the introduction, step~1 
(the enumeration of vertices) is typically the computational bottleneck.
Because so few vertices are admissible, it is
crucial for normal surface enumeration algorithms to enforce the quadrilateral
constraints as they go---one cannot afford to enumerate the entire vertex
set of $\qproj$ (which is generally orders of magnitude larger)
and then enforce the quadrilateral constraints afterwards.

Traditionally, normal surface enumeration algorithms are based on the
double description method of Motzkin et~al.\ \cite{motzkin53-dd},
with additional algorithmic improvements specific to normal surface theory
\cite{burton10-dd}.  In brief, we construct a sequence of polytopes
$P_0,P_1,\ldots,P_e$, where $P_0$ is the unit
simplex in $\R^{3n}$, $P_e$ is the final solution space $\qproj$, and
each intermediate polytope $P_i$ is the intersection of the unit simplex
with the first $i$ matching equations.  The algorithm works inductively,
constructing $P_{i+1}$ from $P_i$ at each stage.  For each intermediate
polytope $P_i$ we throw away any vertices that do not satisfy the
quadrilateral constraints, and it can be shown inductively that this
yields the correct final solution set.

It is worth noting why backtracking and reverse search algorithms
have not been used for normal surface enumeration to date:
\begin{itemize}
    \item Reverse search algorithms map out vertices by pivoting along
    edges of the polytope.  This makes it difficult to incorporate the
    quadrilateral constraints: since the region that satisfies these
    constraints can be disconnected, we may need to pivot through
    vertices that \emph{break} these constraints in order to map out all
    solutions.
    Degeneracy can also cause significant performance problems for reverse
    search algorithms \cite{avis00-revised,avis97-howgood-compgeom}.
    \item Backtracking algorithms receive comparatively little attention
    in the literature.  Fukuda et~al.\ \cite{fukuda97-analysis}
    show that backtracking can solve the \emph{face} enumeration problem
    in polynomial time.  However, for vertex enumeration their
    framework requires a solution to the NP-complete \emph{restricted
    vertex problem}, and they conclude that straightforward backtracking
    is unlikely to work efficiently for vertex enumeration in general.
\end{itemize}

Our new algorithm belongs to the backtracking family.
Despite the concerns of Fukuda et~al., we find in
Section~\ref{s-perf} that for large problems our algorithm is a
significant improvement over the prior state of the art, often running
orders of magnitude faster.
We achieve this by introducing \emph{domination tests}, which are
time and space trade-offs that allow us to avoid solving the
restricted vertex problem directly.
The full algorithm is given in Section~\ref{s-algm} below, and
in Sections~\ref{s-complexity} and~\ref{s-perf} we study
these trade-offs both theoretically and experimentally.

\section{The tree traversal algorithm} \label{s-algm}

The basic idea behind the algorithm is to construct admissible vertices
$\mathbf{x} \in \qproj$ by iterating through all possible
combinations of which coordinates
$x_i$ are zero and which coordinates $x_i$ are non-zero.
We arrange these combinations into a search tree of height $n$,
which we traverse in a depth-first manner.
Using a combination of feasibility tests and domination tests, we are
able to prune this search tree so that the traversal is more efficient,
and so that the leaves at depth $n$
correspond precisely to the admissible vertices of $\qproj$.

Because the quadrilateral constraints are expressed purely in terms
of zero versus non-zero coordinates, we can easily build the quadrilateral
constraints directly into the structure of the search tree.
We do this with the help of \emph{type vectors}, which we introduce
in Section~\ref{s-algm-type}.  In Section~\ref{s-algm-main} we describe
the search tree and present the overall structure of the
algorithm, and we follow in Sections~\ref{s-algm-domination}
and~\ref{s-algm-feasibility} with details on some of the more complex steps.

\subsection{Type vectors} \label{s-algm-type}

Recall the quadrilateral constraints, which state that
for each $i=1,\ldots,n$, at most one of the three coordinates
$x_{3i-2},x_{3i-1},x_{3i}$ can be non-zero.  A \emph{type vector}
is a sequence of $n$ symbols that indicates
how these constraints are resolved for each $i$.  In detail:

\begin{defn}[Type vector] \label{d-type}
    Let $\mathbf{x} = (x_1,\ldots,x_{3n}) \in \R^{3n}$ be any vector
    that satisfies the quadrilateral constraints.
    We define the \emph{type vector} of $\mathbf{x}$ to be
    $\tau(\mathbf{x}) = (\tau_1,\ldots,\tau_n) \in \{0,1,2,3\}^n$,
    where
    \begin{equation} \label{eqn-type}
    \tau_i = \left\{\begin{array}{ll}
        0 & \mbox{if $x_{3i-2} = x_{3i-1} = x_{3i} = 0$}; \\
        1 & \mbox{if $x_{3i-2} \neq 0$ and $x_{3i-1} = x_{3i} = 0$}; \\
        2 & \mbox{if $x_{3i-1} \neq 0$ and $x_{3i-2} = x_{3i} = 0$}; \\
        3 & \mbox{if $x_{3i} \neq 0$ and $x_{3i-2} = x_{3i-1} = 0$}. \\
    \end{array}\right.
    \end{equation}
\end{defn}

For example, consider the vector
$\mathbf{x} = (0,1,0,\ 0,0,4,\ 0,0,0,\ 0,0,2) \in \R^{12}$
where $n=4$.
Here the type vector of $\mathbf{x}$ is $\tau(\mathbf{x}) = (2,3,0,3)$.

An important feature of type vectors is that they carry enough
information to completely reconstruct any vertex of the projective
solution space, as shown by the following result.

\begin{lemma} \label{l-type-reconstruct}
    Let $\mathbf{x}$ be any admissible vertex of $\qproj$.
    Then the precise coordinates $x_1,\ldots,x_{3n}$ 
    can be recovered from the type vector $\tau(\mathbf{x})$
    by solving the following simultaneous equations:
    \begin{itemize}
        \item the matching equations $M \mathbf{x} = \mathbf{0}$;
        \item the projective equation $\sum x_i = 1$;
        \item the equations of the form $x_j=0$ as dictated by
        (\ref{eqn-type}) above, according to the particular
        value (0, 1, 2 or 3) of each entry in the type vector
        $\tau(\mathbf{x})$.
    \end{itemize}
\end{lemma}

\begin{proof}
    This result simply translates Lemma~4.4 of \cite{burton10-dd} into the
    language of type vectors.  However, it is also a simple consequence
    of polytope theory: the type vector indicates which facets of
    $\qproj$ the point $\mathbf{x}$ belongs to, and any vertex of a
    polytope can be reconstructed as the intersection of those
    facets to which it belongs.  See \cite{burton10-dd} for further details.
\end{proof}

Note that this full recovery of $\mathbf{x}$ from $\tau(\mathbf{x})$
is only possible for \emph{vertices} of $\qproj$, not arbitrary points
of $\qproj$.  In general, the type vector $\tau(\mathbf{x})$ carries
only enough information for us to determine the smallest face of $\qproj$
to which $\mathbf{x}$ belongs.

However, type vectors do carry enough information for us to
identify \emph{which} admissible points $\mathbf{x} \in \qproj$ are
vertices of $\qproj$.
For this we introduce the concept of domination.

\begin{defn}[Domination]
    Let $\tau, \sigma \in \{0,1,2,3\}^n$ be type vectors.
    We say that $\tau$ \emph{dominates} $\sigma$ if,
    for each $i=1,\ldots,n$, either $\sigma_i=\tau_i$ or $\sigma_i=0$.
    We write this as $\tau \geq \sigma$.
\end{defn}

For example, if $n=4$ then
$(1,0,2,3) \geq (1,0,2,0) \geq (1,0,2,0) \geq (1,0,0,0)$.
On the other hand, neither of $(1,0,2,0)$ or $(1,0,3,0)$ dominates the other.
It is clear that in general, domination imposes a partial order
(but not a total order) on type vectors.

\begin{remark}
    We use the word ``domination'' because domination of type vectors
    corresponds precisely to domination in the face lattice of the polytope
    $\qproj$.
    For any two admissible points $\mathbf{x},\mathbf{y} \in \R^{3n}$,
    we see from
    Definition~\ref{d-type} that $\tau(\mathbf{x}) \geq \tau(\mathbf{y})$
    if and only if, for every coordinate position $i$ where $x_i=0$,
    we also have $y_i=0$.  Phrased in the language of polytopes, this means
    that $\tau(\mathbf{x}) \geq \tau(\mathbf{y})$ if and only if
    every face of $\qproj$ containing $\mathbf{x}$ also contains $\mathbf{y}$.
    That is, $\tau(\mathbf{x}) \geq \tau(\mathbf{y})$ if and only if
    the smallest-dimensional face containing $\mathbf{y}$ is a \emph{subface}
    of the smallest-dimensional face containing $\mathbf{x}$.
\end{remark}

These properties allow us to characterise admissible vertices of the
projective solution space entirely in terms of type vectors,
as seen in the following result.

\begin{lemma} \label{l-type-dominate}
    Let $\mathbf{x} \in \qproj$ be admissible.  Then the
    following statements are equivalent:
    \begin{itemize}
        \item $\mathbf{x}$ is a vertex of $\qproj$;
        \item there is no admissible point $\mathbf{y} \in \qproj$ for which
        $\tau(\mathbf{x}) \geq \tau(\mathbf{y})$ and
        $\mathbf{x} \neq \mathbf{y}$;
        \item there is no admissible vertex $\mathbf{y}$ of $\qproj$ for which
        $\tau(\mathbf{x}) \geq \tau(\mathbf{y})$ and
        $\mathbf{x} \neq \mathbf{y}$.
    \end{itemize}
\end{lemma}

\begin{proof}
    This is essentially Lemma~4.3 of \cite{burton09-convert},
    whose proof involves elementary polytope theory with some minor
    complications due to the quadrilateral constraints.
    Here we merely reformulate this result in terms of type vectors.
    See \cite{burton09-convert} for further details.
\end{proof}

Because our algorithm involves backtracking, it spends much of its time
working with partially-constructed type vectors.  It is therefore useful to
formalise this notion.

\begin{defn}[Partial type vector]
    A \emph{partial type vector} is any vector $\tau = (\tau_1,\ldots,\tau_n)$,
    where each $\tau_i \in \{0,1,2,3,\unk\}$.
    We call the special symbol ``$\unk$'' the \emph{unknown symbol}.
    A partial type vector that does not contain any unknown symbols is
    also referred to as a \emph{complete type vector}.

    We say that two partial type vectors $\tau$ and $\sigma$ \emph{match}
    if, for each $i=1,\ldots,n$, either $\tau_i = \sigma_i$ or at least
    one of $\tau_i,\sigma_i$ is the unknown symbol.
\end{defn}

For example, $(1,\unk,2)$ and $(1,3,2)$ match,
and $(1,\unk,2)$ and $(1,3,\unk)$ also match.  However,
$(1,\unk,2)$ and $(0,\unk,2)$ do not match because their leading entries
conflict.  If $\tau$ and $\sigma$ are both complete type vectors, then
$\tau$ matches $\sigma$ if and only if $\tau = \sigma$.

\begin{defn}[Type constraints] \label{d-type-constraints}
    Let $\tau = (\tau_1,\ldots,\tau_n) \in \{0,1,2,3,\unk\}^n$ be a
    partial type vector.  The \emph{type constraints} for $\tau$ are a
    collection of equality constraints and non-strict inequality
    constraints on an arbitrary point
    $\mathbf{x} = (x_1,\ldots,x_{3n}) \in \R^{3n}$.
    For each $i=1,\ldots,n$, the type symbol $\tau_i$ contributes the
    following constraints to this collection:
    \[ \begin{array}{ll}
    x_{3i-2} = x_{3i-1} = x_{3i} = 0 &
        \mbox{if $\tau_i = 0$;} \\
    x_{3i-2} \geq 1\ \textrm{and}\ x_{3i-1} = x_{3i} = 0 &
        \mbox{if $\tau_i = 1$;} \\
    x_{3i-1} \geq 1\ \textrm{and}\ x_{3i-2} = x_{3i} = 0 &
        \mbox{if $\tau_i = 2$;} \\
    x_{3i} \geq 1\ \textrm{and}\ x_{3i-2} = x_{3i-1} = 0 &
        \mbox{if $\tau_i = 3$;} \\
    \textrm{no constraints} &
        \mbox{if $\tau_i = \unk$.}
    \end{array}\]
\end{defn}

Note that, unlike all of the other constraints seen so far, the type
constraints are not invariant under scaling.  The type constraints are
most useful in the solution cone $\qcone$, where any admissible
point $\mathbf{x} \in \qcone$ with $x_j > 0$ can be rescaled to some
admissible point $\lambda \mathbf{x} \in \qcone$ with $\lambda x_j \geq 1$.
We see this scaling behaviour in Lemma~\ref{l-type-constraints} below.

The type constraints are similar but not identical to the conditions
described by equation~(\ref{eqn-type}) for the type vector $\tau(\mathbf{x})$,
and their precise relationship is described by the following lemma.
The most important difference
is that (\ref{eqn-type}) uses strict inequalities of the form $x_j > 0$,
whereas the type constraints use non-strict inequalities of the form
$x_j \geq 1$.  This is because the type constraints will be used with
techniques from linear programming, where non-strict inequalities are
simpler to deal with.

\begin{lemma} \label{l-type-constraints}
    Let $\sigma \in \{0,1,2,3,\unk\}^n$ be any partial type vector.
    If $\mathbf{x} \in \R^{3n}$ satisfies the type constraints for
    $\sigma$, then the type vector $\tau(\mathbf{x})$ matches $\sigma$.
    Conversely, if $\mathbf{x} \in \R^{3n}$ is an admissible vector whose
    type vector $\tau(\mathbf{x})$ matches $\sigma$, then
    $\lambda \mathbf{x}$ satisfies the type constraints for $\sigma$ for
    some scalar $\lambda > 0$.
\end{lemma}

\begin{proof}
    This is a simple exercise in matching Definitions~\ref{d-type}
    and~\ref{d-type-constraints}.
    Recall that admissible vectors $\mathbf{x}$ must satisfy
    $\mathbf{x} \geq \mathbf{0}$, which is why we are able to convert
    $x_j \neq 0$ into $\lambda x_j \geq 1$.
\end{proof}

We finish this section on type vectors with a simple but important note.

\begin{lemma} \label{l-type-nonzero}
    In the projective solution space, type vectors are always non-zero.
    That is, there is no $\mathbf{x} \in \qproj$
    for which $\tau(\mathbf{x}) = \mathbf{0}$.
\end{lemma}

\begin{proof}
    If $\tau(\mathbf{x})=\mathbf{0}$ then $\mathbf{x}=\mathbf{0}$,
    by Definition~\ref{d-type}.  However, every point in the projective
    solution space lies in the hyperplane $\sum x_i = 1$.
\end{proof}


\subsection{The structure of the algorithm} \label{s-algm-main}

Recall that our overall plan is to iterate through all possible
combinations of zero and non-zero coordinates.  We do this by iterating
through all possible type vectors, thereby implicitly enforcing the
quadrilateral constraints as we go.  The framework for this iteration is
the \emph{type tree}, which we define as follows.

\begin{defn}[Type tree]
    The \emph{type tree} is a rooted tree of height $n$, where
    all leaf nodes are at depth $n$, and where each non-leaf node has
    precisely four children.  The four edges descending from each
    non-leaf node are labelled 0, 1, 2 and 3 respectively.

    Each node is labelled with a partial type vector.  The root node is
    labelled $(\unk,\ldots,\unk)$.  Each non-leaf node at depth $i$
    has a label of the form $(\tau_1,\ldots,\tau_i,\unk,\unk,\ldots,\unk)$,
    and its children along edges 0, 1, 2 and 3 have labels
    $(\tau_1,\ldots,\tau_i,0,\unk,\ldots,\unk)$,
    $(\tau_1,\ldots,\tau_i,1,\unk,\ldots,\unk)$,
    $(\tau_1,\ldots,\tau_i,2,\unk,\ldots,\unk)$ and
    $(\tau_1,\ldots,\tau_i,3,\unk,\ldots,\unk)$ respectively.
    Figure~\ref{fig-typetree} illustrates the type tree of height $n=2$.
\end{defn}

\begin{figure}[htb]
    \centering
    \includegraphics[scale=0.9]{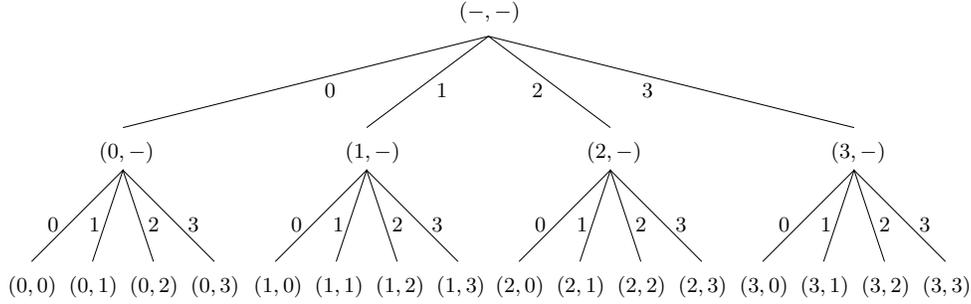}
    \caption{The type tree of height $n=2$}
    \label{fig-typetree}
\end{figure}

The algorithm walks through the type tree in a depth-first manner,
collecting admissible vertices of $\qproj$ as it goes.  The tree is
large however, with $O(4^n)$ nodes, and so we do not traverse the type
tree in full.  Instead we prune the tree so that we only follow edges
that might lead to admissible vertices.  The main tools that we use
for pruning are feasibility tests (based on the type constraints)
and domination tests (based on the previous vertices found so far).
The details are as follows.

\begin{algorithm}[Tree traversal algorithm] \label{a-tree}
    The following algorithm takes a 3-manifold triangulation as input,
    and outputs the set of all admissible vertices of the projective
    solution space $\qproj$.

    Construct the matching equations $M$ (see
    \cite{burton09-convert} for details), and delete rows until $M$ has
    full rank.  Initialise an empty set $V$; this will be used to
    store the type vectors for admissible vertices of $\qproj$.

    Beginning at the root of the type tree, process nodes recursively
    according to the following instructions.  In general, we process
    the node $N$ as follows:

    \begin{itemize}
        \item If $N$ is a non-leaf node then examine the four children of
        $N$ in turn, beginning with the child along edge $0$ (the order
        of the remaining children is unimportant).  For a child node
        labelled with the partial type vector $\tau$, recursively
        process this child if and only if all of the following
        conditions are satisfied:
        \begin{enumerate}[(a)]
            \item \emph{Zero test:} $\tau$ is not the zero vector.
            \item \emph{Domination test:}
            If we replace every unknown symbol $\unk$ with
            $0$, then $\tau$ does not dominate any type vector in $V$.
            \item \emph{Feasibility test:}
            There is some point $\mathbf{x} \in \R^{3n}$ that
            satisfies $\mathbf{x} \geq \mathbf{0}$, $M\mathbf{x} = \mathbf{0}$,
            and the type constraints for $\tau$.
        \end{enumerate}

        \item If $N$ is a leaf node then its label is a complete
        type vector $\tau$, and we claim that $\tau$ must in fact be
        the type vector of an admissible vertex of $\qproj$
        (we prove this in Theorem~\ref{t-algm}).
        Insert $\tau$ into the set $V$,
        reconstruct the full vertex $\mathbf{x} \in \R^{3n}$ from $\tau$
        as described by Lemma~\ref{l-type-reconstruct}, and output this vertex.
    \end{itemize}
\end{algorithm}

Both the domination test and the feasibility test are non-trivial.
We discuss the details of these tests in Sections~\ref{s-algm-domination}
and~\ref{s-algm-feasibility} below,
along with the data structure used to store the solution set $V$.
On the other hand, the zero test is simple; moreover, it is easy to
see that it only needs to be tested once (for the very first leaf node
that we process).

From the viewpoint of normal surface theory, it is important that this
algorithm be implemented using \emph{exact arithmetic}.  This is
because each admissible vertex $\mathbf{x} \in \qproj$ needs to be scaled
to its smallest integer multiple $\lambda \mathbf{x} \in \Z^{3n}$ in
order to extract useful information for topological applications.

This is problematic because both the numerators and the denominators
of the rationals that we encounter can grow exponentially large in $n$.
For a na{\"i}ve implementation we could simply use arbitrary-precision
rational arithmetic (as provided for instance by the GMP library \cite{gmp}).
However, the bounds that we prove in Section~\ref{s-bounds} of
this paper allow us to use native machine integer
types (such as 32-bit, 64-bit or 128-bit integers) for many reasonable-sized
applications.

\begin{theorem} \label{t-algm}
    Algorithm~\ref{a-tree} is correct.  That is,
    the output is precisely the set of all admissible vertices of $\qproj$.
\end{theorem}

\begin{proof}
    By Lemma~\ref{l-type-reconstruct} the algorithm is
    correct if, when it terminates, the set $V$ is precisely the
    set of type vectors of all admissible vertices of $\qproj$.  We prove
    this latter claim in three stages.

    \begin{enumerate}
        \item \emph{Every type vector in $V$ is the type vector of some
        admissible point $\mathbf{x} \in \qproj$.}
        \label{en-algm-adm}

        Let $\sigma$ be some type vector in $V$.  Because we reached the
        corresponding leaf node in our tree traversal,
        we know that $\sigma$ passes the
        feasibility test.  That is, there is some $\mathbf{x} \in \R^{3n}$
        that satisfies $\mathbf{x} \geq 0$, $M\mathbf{x} = \mathbf{0}$,
        and the type constraints for $\sigma$.  Moreover, $\sigma$
        also passes the zero test, which means we can rescale
        $\mathbf{x}$ so that $\sum x_i = 1$.

        Because $\sigma$ was obtained from a leaf node it must be a
        complete type vector, which means that the type constraints for
        $\sigma$ also enforce the quadrilateral constraints.  In other words,
        this point $\mathbf{x}$ must be an admissible point in $\qproj$.

        From Lemma~\ref{l-type-constraints} we know that
        $\tau(\mathbf{x})$ matches $\sigma$, and because $\sigma$ is
        complete it follows that $\tau(\mathbf{x}) = \sigma$.
        That is, $\sigma$ is the type vector of an admissible point in $\qproj$.

        \item \emph{For every admissible vertex $\mathbf{x}$ of $\qproj$,
        the type vector $\tau(\mathbf{x})$ appears in $V$.}
        \label{en-algm-vtot}

        All we need to show here is that we can descend from the root of the
        type tree to the leaf node labelled $\tau(\mathbf{x})$ without being
        stopped by the zero test, the domination test or the feasibility
        test.

        Let $N$ be any ancestor of the leaf labelled $\tau(\mathbf{x})$
        (possibly including
        the leaf itself), and let $\sigma$ be its label.  This means
        that $\sigma$ can be obtained from $\tau(\mathbf{x})$ by replacing
        zero or more entries with the unknown symbol.  In particular,
        $\sigma$ matches $\tau(\mathbf{x})$.

        We know that $\sigma$ passes the zero test by
        Lemma~\ref{l-type-nonzero}.
        From Lemma~\ref{l-type-constraints} there is some
        $\lambda > 0$ for which $\lambda\mathbf{x}$ satisfies
        the type constraints for $\sigma$.
        Since $\mathbf{x} \in \qproj$
        it is clear also that $\lambda\mathbf{x} \geq \mathbf{0}$
        and $M \lambda\mathbf{x} = \mathbf{0}$,
        and so $\sigma$ passes the feasibility test.

        By Lemma~\ref{l-type-dominate} the type vector
        $\tau(\mathbf{x})$ does not dominate $\tau(\mathbf{y})$
        for any admissible point $\mathbf{y} \in \qproj$, and from
        stage~\ref{en-algm-adm} above it follows that
        $\tau(\mathbf{x})$ does not dominate any type vector in $V$.
        The same is still true if we replace some elements of $\tau(\mathbf{x})$
        with $0$ (since this operation cannot introduce new dominations),
        which means that $\sigma$ passes the domination test.

        \item \emph{Every type vector in $V$ is the
        type vector of some admissible vertex $\mathbf{x} \in \qproj$.}
        \label{en-algm-ttov}

        Suppose that some $\sigma \in V$ is not the type vector
        of an admissible vertex.  By stage~\ref{en-algm-adm} above
        we know there is \emph{some} admissible point $\mathbf{x} \in \qproj$
        for which $\sigma = \tau(\mathbf{x})$.  Because $\mathbf{x}$ is
        not a vertex, it follows from Lemma~\ref{l-type-dominate}
        that $\sigma \geq \tau(\mathbf{y})$
        for some admissible vertex $\mathbf{y} \in \qproj$.

        By stage~\ref{en-algm-vtot} above we know that
        $\tau(\mathbf{y})$ also appears in $V$.  Since
        $\sigma \geq \tau(\mathbf{y})$, the type vector
        $\tau(\mathbf{y})$ can be obtained from $\sigma$
        by replacing one or more entries with $0$.  This means that the
        algorithm processes the leaf labelled $\tau(\mathbf{y})$
        \emph{before} the leaf labelled $\sigma$, yielding a contradiction:
        the node labelled $\sigma$ should never have been processed, since
        it must have failed the domination test.
    \end{enumerate}

    Steps \ref{en-algm-vtot} and \ref{en-algm-ttov} together show that
    $V$ is precisely the set of type vectors of all admissible vertices
    of $\qproj$, and the algorithm is therefore proven correct.
\end{proof}

\begin{remark}
    All three tests (zero, domination and feasibility) can be performed
    in polynomial time in the input and/or output size.  What prevents
    the entire algorithm from running in polynomial time is \emph{dead ends}:
    nodes that we process but which do not eventually lead to any
    admissible vertices.

    It is worth noting that
    Fukuda et~al.\ \cite{fukuda97-analysis} describe a backtracking
    framework that does not suffer from dead ends.  However, this comes
    at a significant cost: before processing each node they must solve the
    \emph{restricted vertex problem}, which essentially asks whether a
    vertex exists beneath a given node in the search tree.  They prove
    this restricted vertex problem to be NP-complete, and they do not give an
    explicit algorithm to solve it.
\end{remark}

We discuss time and space complexity further in Section~\ref{s-complexity},
where we find that---despite the presence of dead ends---this tree
traversal algorithm enjoys significantly lower worst-case complexity bounds
than the prior state of the art.  This is supported with experimental
evidence in Section~\ref{s-perf}.
In the meantime, it is worth noting some additional advantages of the
tree traversal algorithm:
\begin{itemize}
    \item \emph{The tree traversal algorithm lends itself well to
    parallelisation.}

    The key observation here is that, for each non-leaf node $N$,
    the children along edges $1$, $2$ and $3$ can be processed
    simultaneously (though edge $0$ still needs to be processed first).
    There is no need for these three processes
    to communicate, since they cannot affect each other's domination tests.
    This three-way branching can be carried out repeatedly as we move
    deeper through the type tree, allowing us to make effective use of a
    large number of processors if they are available.
    Distributed processing is also practical, since at each branch the
    data transfer has small polynomial size.

    \item \emph{The tree traversal algorithm supports incremental output.}

    If the type vectors of admissible vertices are reasonably well
    distributed across the type tree (as opposed to tightly clustered in a
    few sections of the tree), then we can expect a regular stream of
    output as the algorithm runs.  From practical experience, this is indeed
    what we see: for problems that run for hours or even days, the algorithm
    outputs solutions at a continual (though slowing) pace, right from the
    beginning of the algorithm through until its termination.

    This incremental output is important for two reasons:

    \begin{itemize}
        \item It allows \emph{early termination} of the algorithm.
        For many topological decision algorithms, our task is to find
        \emph{any} admissible vertex with some property $P$ (or else
        conclude that no such vertex exists).
        As the tree traversal algorithm outputs vertices we can
        test them immediately for property $P$, and as soon as such a
        vertex is found we can terminate the algorithm.

        \item It further assists with \emph{parallelisation} for
        problems in which testing for property $P$ is expensive
        (this is above and beyond the parallelisation
        techniques described above).
        An example is the Hakenness testing problem, where the test for
        property $P$ can be far more difficult than the original
        normal surface enumeration \cite{burton12-ws}.  For problems
        such as these, we can have a main tree traversal process that
        outputs vertices into a queue for processing,
        and separate dedicated processes that simultaneously work through
        this queue and test vertices for property $P$.
    \end{itemize}

    This is a significant improvement over the double description method
    (the prior state of the art for normal surface enumeration).
    As outlined in Section~\ref{s-prelim}, the double description method
    works inductively through a series of stages, and it typically does
    not output any solutions at all until it reaches the final stage.

    \item \emph{The tree traversal algorithm supports progress tracking.}

    For any given node in the type tree, it is simple to estimate when it
    appears (as a percentage of running time) in a depth-first traversal
    of the full tree.  We can present these estimates to the user as the
    tree traversal algorithm runs, in order to give them some indication of
    the running time remaining.

    Of course such estimates are inaccurate: they ignore pruning (which
    allows us to avoid significant portions of the type tree), and
    they ignore the variability in running times for the feasibility and
    domination tests (for instance, the domination test may slow down
    as the algorithm progresses and the solution set $V$ grows).
    Nevertheless, they provide a rough indication of how far through the
    algorithm we are at any given time.

    Again this is a significant improvement over the double description method.
    Some of the stages in the double description method
    can run many orders of magnitude slower than others
    \cite{avis97-howgood-compgeom,burton09-convert},
    and it is difficult to estimate in advance how slow each stage
    will be.  In this sense,
    the obvious progress indicators (i.e., which stage we are up to and
    how far through it we are) are extremely poor indicators of global
    progress through the double description algorithm.
\end{itemize}

\subsection{Implementing the domination test} \label{s-algm-domination}

What makes the domination test difficult is that
the solution set $V$ can grow exponentially large \cite{burton10-complexity}.
For closed triangulations, the best known theoretical bound is
$|V| \in O\left(\left[\frac{3+\sqrt{13}}{2}\right]^n\right)
\simeq O\left(3.303^n\right)$ as proven in \cite{burton11-asymptotic},
and for bounded triangulations there is only the trivial bound
$|V| \leq 4^n$ (the total number of leaves in the type tree).

The central operation in the domination test is to decide, given a complete%
\footnote{Recall that for the domination test we replace every unknown
symbol $\unk$ with $0$.}
type vector $\tau$, whether there exists some $\sigma \in V$ for which
$\tau \geq \sigma$.
A na{\"i}ve implementation simply walks through the entire set $V$,
giving a time complexity of $O(n|V|)$ (since each individual test of
$\tau \geq \sigma$ runs in $O(n)$ time).

However, we can do better with an efficient data
structure for $V$.  A key observation is that, although we might have
up to $4^n$ type vectors in $V$, there are at most $2^n$ candidate
type vectors $\sigma$ that could possibly be dominated by $\tau$.
More precisely, suppose that $\tau$ contains $k$ non-zero entries and
$n-k$ zero entries (where $0 \leq k \leq n$).  Then the only possible
type vectors that $\tau$ could dominate are those obtained by
replacing some of the non-zero entries in $\tau$ with zeroes,
yielding $2^k$ such candidates in total.

We could therefore store $V$ using a data structure that supports fast
insertion and fast lookup (such as a red-black tree),
and implement the domination test by enumerating all $2^k$ candidate
vectors $\sigma$ and searching for each of them in $V$.  This has a time
complexity of $O(n 2^k \log|V|)$, which can be simplified to $O(n^2\,2^k)$.

Although this looks better than the na{\"i}ve implementation in theory,
it can be significantly worse in
practice.  This is because, although $|V|$ has a theoretical upper bound
of either $O(3.303^n)$ or $4^n$, in practice it is far smaller.
Detailed experimentation on closed triangulations \cite{burton10-complexity}
suggests that on average $|V|$ grows at a rate below $1.62^n$,
which means that enumerating all candidate vectors $\sigma$ can be
far worse than simply testing every $\sigma \in V$.%

A compromise is possible that gives the best of both worlds.
We can represent $V$ using a tree structure that
mirrors the type tree but stores only those nodes with descendants in $V$.
This is essentially a trie (i.e., a prefix tree),
and is illustrated in Figure~\ref{fig-trie}.

\begin{figure}[htb]
    \centering
    \includegraphics[scale=0.9]{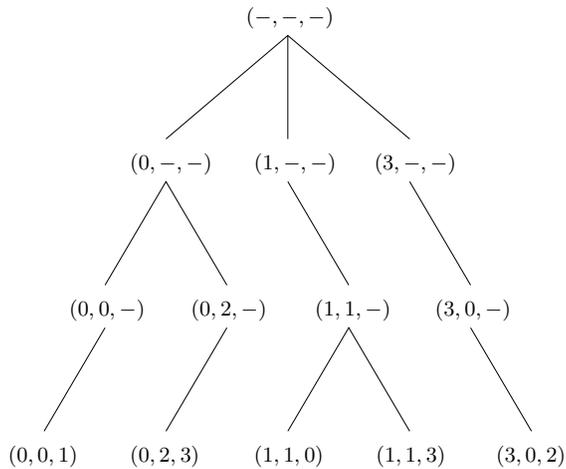}
    \caption{A trie representing
    $V = \{(0,0,1),\ (0,2,3),\ (1,1,0),\ (1,1,3),\ (3,0,2)\}$}
    \label{fig-trie}
\end{figure}

To implement the domination test we simply walk through the portion of the
trie that corresponds to nodes dominated by $\tau$.  This walk spans at most
$2^k$ leaf nodes plus at most $n 2^k$ ancestors.  However, we also know
that the trie contains at most $n|V|$ nodes in total (ancestors included).
The walk therefore covers $O(\min\{n|V|,\,n2^k\})$ nodes,
and since each comparison $\tau \geq \sigma$ comes for free with the
structure of the trie, we have an overall running time of
$O(\min\{n|V|,\,n2^k\})$ for the domination test.

\subsection{Implementing the feasibility test} \label{s-algm-feasibility}

Recall that the feasibility test asks whether any point
$\mathbf{x} \in \R^{3n}$ satisfies
$\mathbf{x} \geq \mathbf{0}$, $M\mathbf{x}=\mathbf{0}$, and the type
constraints for a given partial type vector $\tau$.
Tests of this kind are standard in the linear programming literature,
and can be performed in polynomial time using interior point methods
\cite{deghellinck86-newton,karmarkar84-new}.

However, our implementation of the feasibility test is based not on
interior point methods, but on the \emph{dual simplex method}.
This is a pivot-based algorithm, and although it has a worst-case
exponential running time in theory, it is simple to code and performs
extremely well in practice \cite{megiddo86-selfdual,smale83-simplex}.
Moreover, the dual simplex method makes it easy to add
new constraints to a previously-solved problem, allowing us to
\emph{incrementally} update our solutions $\mathbf{x} \in \R^{3n}$ as
we walk through the type tree.
For details on the dual simplex method, the reader is referred to a
standard text on linear programming such as \cite{bazaraa10-lp}.

Here we outline the framework of the dual simplex method as it applies
to the feasibility test.  We explicitly describe the various matrices
and vectors that appear, since these matrices and vectors play a
key role in the results of Section~\ref{s-bounds}.

Our first step is to simplify the integer matrix of matching equations $M$:

\begin{defn}[Reduced matching equations]
    Let $M$ be the integer matrix of matching equations, as described in
    Section~\ref{s-prelim}.  The \emph{reduced matching equation matrix}
    $\redm$ is obtained from $M$ by dividing each non-zero column by its
    greatest common divisor, which is always taken to be positive.
\end{defn}

An example matrix $M$ and the corresponding reduced matrix $\redm$
are illustrated below.  The first and fourth columns are divided
by $2$, and the remaining columns are left unchanged.
\[ M=\left[{\begin{array}{rrrrrr}
    0&\phantom{-1}\llap{1}&-1&2&-1&-1\\
    -2&0&2&-2&0&2
\end{array}}\right]
\quad \Longrightarrow \quad
\redm=\left[{\begin{array}{rrrrrr}
    \mathbf{0}&\phantom{-1}\llap{1}&-1&\mathbf{1}&-1&-1\\
    \mathbf{-1}&0&2&\mathbf{-1}&0&2
\end{array}}\right] \]

It is clear from Lemma~\ref{l-matching-sparse} that each greatest common
divisor is $1$, $2$, $3$ or $4$, and in the case of an orientable
triangulation just $1$ or $2$.  Although this reduction is small,
it turns out to have a significant impact on the arithmetical bounds
that we prove in Section~\ref{s-bounds}.

Although a solution to $M\mathbf{x}=\mathbf{0}$ need not satisfy
$\redm\mathbf{x}=\mathbf{0}$ (and vice versa), the two
matrices are interchangeable for our purposes as shown by the following
simple lemma.

\begin{lemma} \label{l-feasibility-reduce}
    Replacing $M$ with $\redm$ does not affect the results of the
    feasibility test.  In other words, for any partial type vector $\tau$
    the following two statements are equivalent:
    \begin{itemize}
        \item
        There exists an $\mathbf{x} \in \R^{3n}$ satisfying
        $\mathbf{x} \geq \mathbf{0}$, $M\mathbf{x}=\mathbf{0}$ and the type
        constraints for $\tau$;
        \item
        There exists an $\mathbf{x}' \in \R^{3n}$ satisfying
        $\mathbf{x}' \geq \mathbf{0}$, $\redm\mathbf{x}'=\mathbf{0}$
        and the type constraints for $\tau$.
    \end{itemize}
\end{lemma}

\begin{proof}
    For each $i=1,\ldots,3n$, let $d_i$ denote the greatest common divisor
    of the $i$th column of $M$ (or $1$ if the $i$th column of $M$ is zero).
    In addition, let $D = \lcm(d_1,\ldots,d_{3n})$.

    Suppose that $\mathbf{x} = (x_1,\ldots,x_{3n})$ satisfies
    $\mathbf{x} \geq \mathbf{0}$, $M\mathbf{x}=\mathbf{0}$ and the type
    constraints for $\tau$.  Then it is clear that
    $\mathbf{x}' = (d_1x_1,\ldots,d_{3n}x_{3n})$ satisfies
    $\mathbf{x}' \geq \mathbf{0}$, $\redm\mathbf{x}'=\mathbf{0}$ and the type
    constraints for $\tau$.

    Conversely, suppose that $\mathbf{x}' = (x'_1,\ldots,x'_{3n})$ satisfies
    $\mathbf{x}' \geq \mathbf{0}$, $\redm\mathbf{x}'=\mathbf{0}$ and the type
    constraints for $\tau$.  In this case, we find that
    $\mathbf{x} = (Dx'_1/d_1,\ldots,Dx'_{3n}/d_{3n})$ satisfies
    $\mathbf{x} \geq \mathbf{0}$, $M\mathbf{x}=\mathbf{0}$ and the type
    constraints for $\tau$.
\end{proof}

Our next step is to replace type constraints of the form $x_i \geq 1$
with equivalent constraints of the form $x'_i \geq 0$.

\begin{lemma} \label{l-feasibility-subtract}
    Let $\tau$ be any partial type vector, and let the type constraints
    for $\tau$ be $x_i = 0$ for all $i \in I$ and
    $x_j \geq 1$ for all $j \in J$.
    Then the following two statements are equivalent:
    \begin{itemize}
        \item
        There exists an $\mathbf{x} \in \R^{3n}$ for which
        $\mathbf{x} \geq \mathbf{0}$, $\redm\mathbf{x}=\mathbf{0}$,
        $x_i = 0$ for all $i \in I$ and $x_j \geq 1$ for all $j \in J$;
        \item
        There exists an $\mathbf{x}' \in \R^{3n}$ for which
        $\mathbf{x}' \geq \mathbf{0}$, $\redm\mathbf{x}'=\mathbf{b}$ and
        $x'_i = 0$ for all $i \in I$,
        where $\mathbf{b} = - \sum_{j \in J} \redm_j$,
        and where $\redm_j$ refers to the $j$th column of $\redm$.
    \end{itemize}
\end{lemma}

\begin{proof}
    The two statements are equivalent under the substitution
    $x'_j = x_j - 1$ for $j \in J$ and $x'_j = x_j$ for $j \notin J$.
\end{proof}

Based on Lemmata~\ref{l-feasibility-reduce}
and~\ref{l-feasibility-subtract}, we structure the feasibility test
as follows.

\begin{algorithm}[Framework for feasibility test]
\label{a-feasibility-framework}
To perform the feasibility test, we use the dual simplex method to
search for a solution $\mathbf{x} \in \R^{3n}$ for which
$\mathbf{x} \geq \mathbf{0}$, $\redm \mathbf{x} = \mathbf{b}$
and $x_i = 0$ for all $i \in I$.  Here the
vector $\mathbf{b}$ and the index set $I$ reflect
the partial type vector $\tau$ that we are working with, and we
construct both $\mathbf{b}$ and $I$ incrementally as follows:
\begin{itemize}
    \item When we begin at the root node with
    $\tau = (\unk,\ldots,\unk)$, we set
    $\mathbf{b} = \mathbf{0}$ and $I = \emptyset$.

    \item Suppose we step down from a parent node labelled
    $(\tau_1,\ldots,\tau_k,\unk,\unk,\ldots,\unk)$ to a child node
    labelled $(\tau_1,\ldots,\tau_k,\tau_{k+1},\unk,\ldots,\unk)$.
    \begin{itemize}
        \item If $\tau_{k+1} = 0$ then we insert the three new indices
        $3k+1$, $3k+2$ and $3k+3$ into the set $I$, reflecting the
        additional type constraints $x_{3k+1}=x_{3k+2}=x_{3k+3}=0$.
        \item If $\tau_{k+1} = 1$ then we insert the two new indices
        $3k+2$ and $3k+3$ into $I$ and subtract the column
        $\redm_{3k+1}$ from $\mathbf{b}$, reflecting the
        additional type constraints
        $x_{3k+2}=x_{3k+3}=0$ and $x_{3k+1} \geq 1$.
        \item The cases $\tau_{k+1} = 2$ and $3$ are handled in a
        similar fashion to $\tau_{k+1}=1$ above.
    \end{itemize}
\end{itemize}
\end{algorithm}

It is important to note that this ``step down'' operation involves only
minor changes to $\mathbf{b}$ and $I$.  This means that, instead of
solving each child feasibility problem from scratch, we can use the
solution $\mathbf{x}$ from the parent node as a starting point
for the new dual simplex search when we introduce our new constraints.
We return to this shortly.

Recall from Algorithm~\ref{a-tree} that $M$ is full rank,
and suppose that $\rank M = \rank \redm = k$.
To perform the feasibility test we maintain the following structures
(all of which are common to revised simplex algorithms, and again
we refer the reader to a standard text \cite{bazaraa10-lp} for details):
\begin{itemize}
    \item a \emph{basis} $\basis = \{\basiselt_1,\ldots,\basiselt_k\}
    \subseteq \{1,\ldots,3n\}$, which identifies $k$ linearly
    independent columns of $\redm$;

    \item the $k \times k$ inverse matrix $\basisinv$, where
    $\basismat$ denotes the $k \times k$ matrix formed from columns
    $\basiselt_1,\ldots,\basiselt_k$ of $\redm$.
\end{itemize}

It is useful to think of $\basisinv$ as a matrix of row operations that
we can apply to both $\redm$ and $\mathbf{b}$.  The product $\basisinv\redm$
includes the $k \times k$ identity as a submatrix, and the product
$\basisinv\mathbf{b}$ identifies the elements of a solution vector
$\mathbf{x}$ as follows.

Any basis $\basis$ describes a solution
$\mathbf{x} \in \R^{3n}$ to the equation $\redm\mathbf{x}=\mathbf{b}$,
in which $x_j = 0$ for each $j \notin \basis$,
and where the product $\basisinv\mathbf{b}$
lists the remaining elements $(x_{\basiselt_1},\ldots,x_{\basiselt_k})$.
Such a solution only satisfies $\mathbf{x} \geq \mathbf{0}$ if we have
$\basisinv\mathbf{b} \geq \mathbf{0}$.
We call the basis \emph{feasible} if $\basisinv\mathbf{b} \geq \mathbf{0}$,
and we call it \emph{infeasible} if $\basisinv\mathbf{b} \ngeq \mathbf{0}$.

The aim of the dual simplex method is to find a feasible basis.
It does this by constructing some initial (possibly infeasible) basis,%
\footnote{The dual simplex method usually has an extra requirement
    that this initial basis be \emph{dual feasible}.  For
    our application there is no objective function to optimise, and so
    this extra requirement can be ignored.}
and then repeatedly modifying this basis using \emph{pivots}.  Each pivot
replaces some index in $\basis$ with some other index not in $\basis$,
and corresponds to a simple sequence of row operations that we apply to
$\basisinv$.
Eventually the dual simplex method produces either a feasible basis or a
certificate that no feasible basis exists.
We apply this procedure to our problem as follows:

\begin{algorithm}[Details of feasibility test] \label{a-feasibility-details}
    At the root node of the type tree, we simply construct any initial
    basis and then pivot until we obtain either a feasible basis
    or a certificate that no
    feasible basis exists (in which case the root node fails the
    feasibility test).

    When we step down from a parent node to a child node, we ``inherit''
    the feasible basis from the parent and use this as the initial basis
    for the child.  This may no longer be feasible because $\mathbf{b}$
    might have changed; if not then once again we pivot until the basis
    can be made feasible or we can prove that no feasible basis exists.

    Recall that we enlarge the index set $I$ when we step to a child node,
    which means that we introduce new constraints of the form $x_j = 0$.
    Each constraint $x_j=0$ is enforced as follows:
    \begin{itemize}
        \item If $j$ is not in the basis then we simply eliminate the
        variable $x_j$ from the system (effectively deleting the
        corresponding column from $\redm$).

        \item If $j$ is in the basis then we pivot to remove it
        and eliminate $x_j$ as before.
        If we cannot pivot $j$ out of the basis then some row of the product
        $\basisinv\redm$ must be zero everywhere except for column $j$,
        which means the equation $\basisinv\redm\mathbf{x}=\basisinv\mathbf{b}$
        implies $x_j = z$ for some scalar $z$.
        If $z \neq 0$ then it is impossible to enforce $x_j=0$.
        If $z=0$ then $x_j=0$ is enforced automatically, and there is
        nothing we need to do.
    \end{itemize}

    The child node is deemed to pass the feasibility test if and only if
    we are able to obtain a feasible basis after enforcing all of our
    new constraints.
\end{algorithm}

The procedures described in Algorithm~\ref{a-feasibility-details}
are all standard elements of the dual simplex method, and once more we
refer the reader to a standard text \cite{bazaraa10-lp} for details and proofs.
There are typically many choices available for which pivots to perform;
one must be careful to choose a pivoting rule that avoids cycling, since
the solution cone $\qcone$ is highly degenerate.  We use Bland's rule
\cite{bland77-pivoting} in our implementation.

\section{Arithmetical bounds} \label{s-bounds}

The tree traversal algorithm involves a significant amount
of rational arithmetic---most notably, in the feasibility test
(Section~\ref{s-algm-feasibility}) and the reconstruction of vertices of
$\qproj$ from their type vectors (Lemma~\ref{l-type-reconstruct}).
Moreover, the rationals that we work with can have exponentially
large numerators and denominators, which means that in general we must
use arbitrary precision arithmetic (as provided for instance by the
GMP library \cite{gmp}).

The aim of this section is to derive explicit bounds on the rational numbers
that appear in our calculations.  In particular,
Theorem~\ref{t-matrix-bounds} bounds the elements of the matrix
$\basisinv$ that we manipulate in the feasibility test, and
Corollaries~\ref{c-vertex-bounds} and~\ref{c-vertex-bounds-closed}
bound the coordinates of the final vertices of $\qproj$.

These results are useful because the relevant bounds can be precomputed
at the beginning of the algorithm, and for many reasonable-sized
problems they allow us to work in native machine integer types
(such as 32-bit, 64-bit or 128-bit integers).  These native types
are significantly faster than arbitrary precision arithmetic:
experimentation with the tree traversal algorithm shows a speed increase
of over 10 times.

Furthermore, Corollary~\ref{c-vertex-bounds-closed} is a significant
improvement on the best known bounds for the coordinates of vertices
of $\qproj$.  These have been independently studied by
Hass, Lagarias and Pippenger \cite{hass99-knotnp} and
subsequently by Burton, Jaco, Letscher and Rubinstein; we recount
these earlier results before presenting Corollary~\ref{c-vertex-bounds-closed}
below.

Central to all of the results in this section is the
\emph{bounding constant} $\bconst$.  This is a function of the
reduced matching equation matrix $\redm$, and to obtain the tightest possible
bounds it can be computed at runtime when the tree
traversal algorithm begins.  If global quantities are required then
Lemma~\ref{l-bconst-bounds} and Corollary~\ref{c-bconst-closed}
bound $\bconst$ in terms of $n$ alone.

\begin{defn}[Bounding constant] \label{d-bconst}
    The \emph{bounding constant} $\bconst$ is defined as follows.
    For each column $c=1,\ldots,3n$, compute the Euclidean length of the
    $c$th column of $\redm$; that is,
    $(\sum_i \redm_{i,c}^2)^\frac12$.
    We define $\bconst$ to be the product of the $k$ largest
    of these lengths, where $k = \rank \redm$.
\end{defn}

For example, consider again the matrix
\[ \redm=\left[{\begin{array}{rrrrrr}
    0&\phantom{-1}\llap{1}&-1&1&-1&-1\\
    -1&0&2&-1&0&2
\end{array}}\right], \]
which has rank $k=2$ and column lengths
$1$, $1$, $\sqrt{5}$, $\sqrt{2}$, $1$ and $\sqrt{5}$.
Here the two largest column lengths are both $\sqrt{5}$, and
so $\bconst = \sqrt{5} \cdot \sqrt{5} = 5$.

\begin{lemma} \label{l-bconst-bounds}
    For an arbitrary 3-manifold triangulation,
    $\bconst \leq (\sqrt{10})^{\rank \redm} \leq (\sqrt{10})^{3n}$.
    For an orientable triangulation, we can improve this bound to
    $\bconst \leq (\sqrt{6})^{\rank \redm} \leq (\sqrt{6})^{3n}$.
\end{lemma}

\begin{proof}
    Using Lemma~\ref{l-matching-sparse} and the fact that each column of
    $\redm$ has $\gcd=1$, there are only a few possibilities for the non-zero
    entries in a column of $\redm$.  In the orientable case the only
    options are $(\pm1,\pm1,\pm1,\pm1)$, $(\pm1,\pm1,\pm1)$,
    $(\pm1,\pm1)$, $(\pm1)$, $(\pm2,\pm1,\pm1)$ and $(\pm2,\pm1)$.
    For non-orientable triangulations we have the additional possibility
    of $(\pm3,\pm1)$.

    It follows that each column has Euclidean length $\leq \sqrt{6}$
    in the orientable case and $\leq \sqrt{10}$ otherwise, yielding bounds
    of $\bconst \leq (\sqrt{6})^{\rank \redm}$ and
    $\bconst \leq (\sqrt{10})^{\rank \redm}$ respectively.
    To finish we note that $\redm$ has $\leq 3n$ columns\footnote{%
        When we begin the tree traversal algorithm the matrix $\redm$ has
        \emph{precisely} $3n$ columns, but recall from
        Section~\ref{s-algm-feasibility} that we might delete columns from
        $\redm$ as we move through the type tree.}
    and so $\rank \redm \leq 3n$.
\end{proof}

A result of Tillmann \cite{tillmann08-finite} shows that $\rank M = n$
for any \emph{closed} 3-manifold triangulation, which allows us to
tighten our bounds further:

\begin{corollary} \label{c-bconst-closed}
    For an arbitrary closed 3-manifold triangulation we have
    $\bconst \leq (\sqrt{10})^n$, and
    for a closed orientable 3-manifold triangulation we have
    $\bconst \leq (\sqrt{6})^n$.
\end{corollary}

\begin{remark}
    In fact these bounds are general:
    both $\bconst \leq (\sqrt{6})^n$ for the orientable case
    and $\bconst \leq (\sqrt{10})^n$ otherwise hold for
    triangulations with boundary also.  The argument relies on a proof that
    $\rank M = e - v \leq n$, where $e$ and $v$ are the number of edges
    and vertices of the triangulation that do not lie within the boundary.
    This proof involves topological techniques beyond the scope of this
    paper; see \cite{burton10-extreme} for the full details.
\end{remark}

We can now use $\bconst$ to bound the rational numbers that
appear as we perform the feasibility test
(Algorithms~\ref{a-feasibility-framework} and~\ref{a-feasibility-details}).
Our main tool is
\emph{Hadamard's inequality}, which states that if $A$ is any
square matrix, $|\det A|$ is bounded above by the product of the Euclidean
lengths of the columns of $A$.

\begin{theorem} \label{t-matrix-bounds}
    Let $\basis \subseteq \{1,\ldots,3n\}$ be any basis as described in
    Section~\ref{s-algm-feasibility}, and let $\basisinv$ be the corresponding
    $k \times k$ inverse matrix, where $k = \rank \redm$.
    Then $\basisinv$ can be expressed in the form $\basisinv=N/\Delta$, where
    $\Delta$ is an integer with $|\Delta| \leq \bconst$,
    and where $N$ is an integer matrix with $|n_{i,j}| \leq \bconst$
    for every $i,j$.
\end{theorem}

Before proving this result, it is important to note that it holds for
\emph{any} basis---not just initial bases or feasible bases, but every
basis that we pivot through along the way.  This means that we can use
Theorem~\ref{t-matrix-bounds} to place an upper bound on every integer
that appears in every intermediate calculation, enabling us
to use native machine integer types instead of arbitrary precision
arithmetic when $\bconst$ is of a reasonable size.

For example, consider the Weber-Seifert space, which is mentioned
in the introduction and has been considered one of the benchmarks for
progress in computational topology.  Using the 23-tetrahedron triangulation
described in \cite{burton12-ws} we find that $\bconst=2^{23}$,
which allows us to store the matrix $\basisinv$ using native (and very fast)
32-bit integers.\footnote{%
    Of course we must be careful: for instance, row operations of the
    form $\mathbf{x} \gets \lambda \mathbf{x} + \mu \mathbf{y}$ must be
    performed using 64-bit arithmetic, even though the inputs and outputs
    are guaranteed to fit into 32-bit integers.}

\begin{proof}[Proof of Theorem~\ref{t-matrix-bounds}]
    Using the adjoint formula for a matrix inverse we have
    $\basisinv = N/\Delta$, where $N = \adj \basismat$ and
    $\Delta = \det \basismat$.
    Because $\redm$ is an integer matrix it is clear that $\Delta$ and all of
    the entries $n_{i,j}$ are integers.

    By Hadamard's inequality, $|\Delta|$ is bounded above by the product of
    the Euclidean lengths of the $k$ columns of $\basismat$, and
    it follows from Definition~\ref{d-bconst} that $|\Delta| \leq \bconst$.
    Since each adjoint entry $n_{i,j}$ is a subdeterminant of
    $\basismat$, a similar argument shows that $|n_{i,j}| \leq \bconst$
    for each $i,j$.
\end{proof}

To finish this section, we use Theorem~\ref{t-matrix-bounds} to
bound the coordinates of any admissible vertex $\mathbf{v} \in \qproj$.
More precisely, we bound the coordinates of $\mathbf{u}=\lambda \mathbf{v}$,
where $\mathbf{u} \in \qcone$ is the smallest \emph{integer}
multiple of $\mathbf{v}$.  This smallest integer multiple has a precise
meaning for topologists (essentially, each integer $u_i$ corresponds to a
collection of $u_i$ quadrilaterals embedded in some tetrahedron of the
triangulation; see \cite{burton09-convert,tollefson98-quadspace} for
full details).

The first such bound was due to Hass, Lagarias and Pippenger
\cite{hass99-knotnp}, who proved that $|u_i| \leq 128^n/2$.
Their result applies only to true simplicial complexes (not the
more general triangulations that we allow here),
and under their assumptions
the matching equation matrix $M$ is much simpler (containing only
$0$ and $\pm1$ entries).

The only other bound known to date appears in unpublished notes by
Burton, Jaco, Letscher and Rubinstein (c.\,2001 and referenced in
\cite{jaco02-algorithms-essential}), where it
is shown that $|u_i| \leq (\sqrt{8})^n$ for a one-vertex closed orientable
triangulation (this time generalised triangulations are allowed).%
\footnote{Both of these previously-known bounds
    were derived in standard coordinates ($\R^{7n}$), not quadrilateral
    coordinates ($\R^{3n}$).  However, it is simple to convert between
    coordinate systems \cite{burton09-convert},
    and it can be shown that the upper bounds differ by a factor
    of at most $4n$.}

In Corollary~\ref{c-vertex-bounds} we obtain new bounds on $|u_i|$
in terms of the bounding constant $\bconst$, and in
Corollary~\ref{c-vertex-bounds-closed} we use this to bound $|u_i|$ in terms
of $n$ alone, yielding significant improvements to both the
Hass et~al.\ and Burton et~al.\ bounds outlined above.

\begin{corollary} \label{c-vertex-bounds}
    Let $\mathbf{u}$ be the smallest integer multiple of an admissible
    vertex $\mathbf{v} \in \qproj$.
    Then each coordinate $u_i$ is bounded as follows:
    \begin{itemize}
        \item $|u_i| \leq (4nk+2)\cdot\bconst$ if the triangulation is
        orientable;
        \item $|u_i| \leq (36nk+12)\cdot\bconst$ otherwise,
    \end{itemize}
    where $k = \rank M = \rank \redm$.
\end{corollary}

\begin{proof}
    This proof is a simple (though slightly messy) matter of starting
    with Theorem~\ref{t-matrix-bounds} and following back through the
    various transformations and changes of variable until we reach
    our solution $\mathbf{v} \in \qproj$ to the original matching equations.

    Let $\tau$ be the type vector of the vertex $\mathbf{v}$, and let
    $\mathbf{x} \in \R^{3n}$
    be the corresponding solution to the feasibility test as found by
    the dual simplex method in Algorithm~\ref{a-feasibility-framework}
    (so that $\mathbf{x} \geq \mathbf{0}$ and $\redm\mathbf{x}=\mathbf{b}$).
    Let $\basis$ be the corresponding feasible basis found by the dual
    simplex method, so that every non-zero entry in $\mathbf{x}$
    appears as an entry of the vector $\basisinv\mathbf{b}$.

    From Lemma~\ref{l-feasibility-subtract} we recall that
    $\mathbf{b}$ is the negative sum of at most $n$ columns
    of $\redm$.  Following the same argument used in the proof of
    Lemma~\ref{l-bconst-bounds},
    each non-zero entry in $\redm$ is either $\pm1$, $\pm2$ or $\pm3$,
    and for an orientable triangulation just $\pm1$ or $\pm2$.
    It follows that each $|b_i| \leq 2n$ if our triangulation is
    orientable, and each $|b_i| \leq 3n$ otherwise.

    We return now to the vector $\mathbf{x}$.
    As in Theorem~\ref{t-matrix-bounds}, let $\basisinv=N/\Delta$, where
    $\Delta$ is an integer with $|\Delta| \leq \bconst$,
    and where each $n_{i,j}$ is an integer with
    $|n_{i,j}| \leq \bconst$.
    Every non-zero element of $\mathbf{x}$ appears in the vector
    $\basisinv\mathbf{b}$, and is therefore of the form
    $\sum_{j=1}^k n_{i,j} b_j / \Delta$.
    Combining all of the bounds above, we deduce that
    $\mathbf{x}=\mathbf{x}'/\Delta$, where
    $\mathbf{x}'$ is an integer vector and where each
    $|x'_i| \leq 2nk\bconst$ if our triangulation is orientable, or
    $|x'_i| \leq 3nk\bconst$ otherwise.

    Our next task is to step backwards through
    Lemma~\ref{l-feasibility-subtract}.  Our current vector $\mathbf{x}$
    satisfies the second statement of Lemma~\ref{l-feasibility-subtract};
    let $\mathbf{y}$ be the corresponding solution for the first
    statement, so that
    $\mathbf{y} \geq \mathbf{0}$, $\redm\mathbf{y}=\mathbf{0}$,
    and $\mathbf{y}$ satisfies the type constraints for $\tau$.
    As in the proof of Lemma~\ref{l-feasibility-subtract}, each
    $y_i = x_i$ or $x_i + 1$.  It follows that
    $\mathbf{y}=\mathbf{y}'/\Delta$, where
    $\mathbf{y}'$ is an integer vector and where each
    $|y'_i| \leq 2nk\bconst + \Delta$ if our triangulation is orientable, or
    $|y'_i| \leq 3nk\bconst + \Delta$ otherwise.

    Finally we step backwards through Lemma~\ref{l-feasibility-reduce}.
    Our vector $\mathbf{y}$ satisfies the second statement of
    Lemma~\ref{l-feasibility-reduce}; let $\mathbf{z}$ be the
    corresponding solution for the first statement, so that
    $\mathbf{z} \geq \mathbf{0}$, $M\mathbf{z} = \mathbf{0}$, and
    $\mathbf{z}$ satisfies the type constraints for $\tau$.
    Following the proof of Lemma~\ref{l-feasibility-reduce},
    each $z_i = Dy_i/d_i$, where $d_i$ is the $\gcd$ of the $i$th column
    of $M$, and where $D=\lcm(d_1,\ldots,d_{3n})$.  By
    Lemma~\ref{l-matching-sparse} we have $D \leq 2$ for an orientable
    triangulation and $D \leq 12$ otherwise.  Therefore
    $\mathbf{z}=\mathbf{z}'/\Delta$, where
    $\mathbf{z}'$ is an integer vector and where each
    $|z'_i| \leq 4nk\bconst + 2\Delta$ if our triangulation is orientable, or
    $|z'_i| \leq 36nk\bconst + 12\Delta$ otherwise.

    To conclude, we observe that our vertex $\mathbf{v}$ is a multiple
    of $\mathbf{z}$, and so the entries in the smallest integer
    multiple $\mathbf{u}$ are bounded by the entries in the (possibly
    larger) integer multiple $\mathbf{z}'$.  In particular,
    $|u_i| \leq 4nk\bconst + 2\Delta \leq (4nk+2) \cdot \bconst$
    if our triangulation is orientable, or
    $|u_i| \leq 36nk\bconst + 12\Delta \leq (36nk + 12) \cdot \bconst$
    otherwise.
\end{proof}

We can recast these bounds purely in terms of $n$: for
$\bconst$ we use Lemma~\ref{l-bconst-bounds}, and for the rank $k$
we use Tillmann's theorem that $\rank M = n$ for closed triangulations
\cite{tillmann08-finite}, or the
observation that $\rank M \leq 3n$ for bounded triangulations.
The result is the following.

\begin{corollary} \label{c-vertex-bounds-closed}
    Let $\mathbf{u}$ be the smallest integer multiple of an admissible
    vertex of $\qproj$.
    Then each coordinate $u_i$ is bounded as follows:
    \begin{itemize}
        \item $|u_i| \leq (4n^2+2)\cdot(\sqrt{6})^n$
        if the triangulation is closed and orientable;
        \item $|u_i| \leq (36n^2+12)\cdot(\sqrt{10})^n$
        if the triangulation is closed and non-orientable;
        \item $|u_i| \leq (12n^2+2)\cdot(\sqrt{6})^{3n}$
        if the triangulation is bounded and orientable;
        \item $|u_i| \leq (108n^2+12)\cdot(\sqrt{10})^{3n}$
        if the triangulation is bounded and non-orientable.
    \end{itemize}
\end{corollary}

\begin{remark}
    These bounds are significant improvements over previous results.
    For bounded triangulations, we improve the Hass et~al.\ bound of
    $O(128^n)$ to $O[n^2(\sqrt{10})^{3n}] \simeq O(n^2\, 31.6^n)$, whilst
    removing the requirement for a true simplicial complex.
    For closed orientable triangulations, we improve the Burton et~al.\ bound
    of $O[(\sqrt{8})^n]$ to $O[n^2 (\sqrt{6})^n]$, whilst removing the
    requirement for a one-vertex triangulation.

    As with Corollary~\ref{c-bconst-closed},
    the stronger bounds that we obtain for closed triangulations
    can in fact be shown to hold for bounded triangulations also.
    Again this relies on a proof that $\rank M = e - v \leq n$,
    which calls on techniques beyond the scope of this paper.
    See \cite{burton10-extreme} for details.
\end{remark}

\section{Time and space complexity} \label{s-complexity}

In this section we derive theoretical bounds on the time and space
complexity of normal surface enumeration algorithms.  We compare these
bounds for the double description method (the prior state of the art,
as described in \cite{burton10-dd})
and the tree traversal algorithm that we introduce in this paper.

All of these bounds are driven by exponential factors,
and it is these exponential factors that we focus on here.
We replace any polynomial factors with a generic polynomial $\polyn$;
for instance, a complexity bound of $O(4^n n^3)$ would simply be
reported as $O(4^n \polyn)$ here.
For both algorithms it is easy to show that these polynomial factors
are of small degree.

We assume that all arithmetical operations can be performed
in polynomial time.  For the tree traversal algorithm this is a simple
consequence of Theorem~\ref{t-matrix-bounds}, and for the double
description method it can likewise be shown that all integers have
$O(n)$ bits.  We omit the details here.

Theorem~\ref{t-complexity-dd} analyses the double description method,
and Theorem~\ref{t-complexity-tree} studies the tree traversal
algorithm.  In summary, the double description method can be performed
in $O(16^n \polyn)$ time and $O(4^n \polyn)$ space, whereas the tree
traversal algorithm can be carried out in
$O(4^n|V|\polyn)$ time and $O(|V|\polyn)$ space,
where $|V|$ is the output size.  This is a significant reduction,
given that the output size for normal surface enumeration is often
extremely small in practice.\footnote{%
    Early indications of this appear in \cite{burton10-complexity}
    (which works in the larger space $\R^{7n}$),
    and a detailed study will appear in \cite{burton10-extreme}.
    To illustrate how extreme these results are in $\R^{3n}$:
    across all $139\,103\,032$ closed 3-manifold triangulations of
    size $n=9$, the maximum output size is just $37$, and the
    \emph{mean} output size is a mere $9.7$.}
If we reformulate these bounds in terms of $n$ alone, the tree traversal
algorithm requires $O(7^n \polyn)$ time, and approximately $O(3.303^n \polyn)$
space for closed triangulations or $O(4^n\polyn)$ space for bounded
triangulations.%

We do not give bounds in terms of $|V|$ for the double description
method, because it suffers from a well-known \emph{combinatorial
explosion}: even when the output size is small, the intermediate
structures that appear can be significantly (and exponentially) more
complex.  See \cite{avis97-howgood-compgeom} for examples of this in
theory, or \cite{burton09-convert} for experiments that show this
in practice for normal surface enumeration.
The result is that, even for bounded triangulations where both
algorithms have space complexities of $O(4^n\polyn)$, the tree traversal
algorithm can effectively exploit a small output size whereas the double
description method cannot.

In conclusion, both the time and space bounds for tree traversal are
notably smaller, particularly when the output size is small.
However, it is important to bear in mind that for both algorithms these
bounds may still be far above the ``real'' time and memory usage that we
observe in practice.  For this reason it is important to carry out
real practical experiments, which we do in Section~\ref{s-perf}.

As a final note, for this theoretical analysis we always assume
implementations that give the best possible complexity bounds.
For the tree traversal algorithm we assume interior point methods
for the feasibility test.
For the double description method we assume that adjacency of vertices
is tested using a polynomial-time algebraic rank test, rather than a
simpler combinatorial test that has worst-case exponential time but
excellent performance in practice \cite{burton10-dd,fukuda96-doubledesc}.
For the experiments of Section~\ref{s-perf} we revert to the dual simplex
method and combinatorial test respectively, both of which are known to
perform extremely well in practical settings
\cite{burton10-dd,fukuda96-doubledesc,megiddo86-selfdual,smale83-simplex}.

\begin{theorem} \label{t-complexity-dd}
    The double description method for normal surface enumeration
    can be implemented in worst-case $O(16^n \polyn)$ time and
    $O(4^n \polyn)$ space.
\end{theorem}

\begin{proof}
    As outlined in Section~\ref{s-prelim}, the double description method
    constructs a series of polytopes $P_0,\ldots,P_e \subseteq \R^{3n}$,
    where each $P_i$ is the intersection of the unit simplex with the
    first $i$ matching equations.  There are $e+1 \in O(n)$ such polytopes
    in total, and with an algebraic adjacency test we can construct each
    $P_i$ in $O(v_{i-1}^2 \polyn)$ time and $O(v_{i-1} \polyn)$ space,
    where $v_{i-1}$ is the number of vertices of $P_{i-1}$ that
    satisfy the quadrilateral constraints \cite{burton10-dd}.

    The main task then is to estimate each $v_i$.
    In every polytope $P_i$, each vertex is uniquely determined by which
    of its coordinates are zero \cite{burton10-dd}.  There are $4^n$
    possible choices of zero coordinates that satisfy the quadrilateral
    constraints, and so each $v_i \leq 4^n$.
    This gives a time complexity of $O(16^n\polyn)$ and
    space complexity of $O(4^n\polyn)$ for the algorithm overall.
\end{proof}

It should be noted that there are other ways of estimating the vertex
counts $v_i$.  For the final polytope, Burton \cite{burton11-asymptotic}
proves that $v_e \in O(3.303^n)$ for a closed triangulation;
however, this result does not hold for
the intermediate vertex counts $v_1,\ldots,v_{e-1}$.
McMullen's theorem \cite{mcmullen70-ubt} gives a tight bound on the
number of vertices of an arbitrary polytope, but here it only yields
$v_i \in O(4.24^n)$; 
the reason it exceeds $4^n$ is that it
does not take into account the quadrilateral constraints.

\begin{theorem} \label{t-complexity-tree}
    The tree traversal algorithm for normal surface enumeration can be
    implemented in worst-case $O(4^n|V|\polyn)$ time and $O(|V|\polyn)$ space,
    where $|V|$ is the number of admissible vertices of $\qproj$.
    In terms of $n$ alone, the worst-case time complexity is
    $O(7^n \polyn)$, and the worst-case space complexity is
    $O\left(\left[\frac{3+\sqrt{13}}{2}\right]^n \polyn\right) \simeq
    O(3.303^n \polyn)$ for closed triangulations or
    $O(4^n\polyn)$ for bounded triangulations.
\end{theorem}

\begin{proof}
    The space complexity is straightforward.
    We do not need to explicitly construct the type tree in memory in
    order to traverse it, and so the only non-polynomial space
    requirements come from storing the solution set $V$.  It follows
    that the tree traversal algorithm has space complexity $O(|V|\polyn)$.
    By the same argument as before we have $|V| \leq 4^n$ for arbitrary
    triangulations, and for closed triangulations
    a result of Burton \cite{burton11-asymptotic} shows that
    $|V| \in O\left(\left[\frac{3+\sqrt{13}}{2}\right]^n \right)$.

    The time complexity is more interesting.
    There are $O(4^n)$ nodes in the type tree, and the
    only non-polynomial operation that we perform on each node is the
    domination test.  As described in Section~\ref{s-algm-domination},
    for a node labelled with the partial type vector $\tau$
    the domination test takes $O(\min\{n|V|,\ n2^k\})$ time, where
    $k$ is the number of times the symbols $1$, $2$ or $3$ appear in $\tau$.
    It follows that the tree traversal algorithm
    runs in $O(4^n|V|\polyn)$ time.

    To obtain a complexity bound that does not involve $|V|$, a simple
    counting argument shows that at most $n \binom{n}{k} 3^k$ nodes in the
    type tree are labelled with a partial type vector $\tau$ with the
    property described above.
    Therefore the total running time is bounded by
    \[ O\left(\left[\sum_{k=0}^n n \binom{n}{k} 3^k \cdot n 2^k \right]
        \polyn \right) =
       O\left(\left[\sum_{k=0}^n \binom{n}{k} 6^k \right] \polyn \right)
       =
       O\left(7^n \polyn\right), \]
    using the binomial expansion
    $\sum_{k=0}^n \binom{n}{k} 6^k = (1+6)^n = 7^n$.
\end{proof}

\section{Experimental performance} \label{s-perf}

In this section we put our algorithm into practice.  We run it through a
test suite of 275 triangulations, based on the first 275 entries in the
census of knot complements tabulated by Christy and shipped with version~1.9
of {\snap} \cite{coulson00-snap}.

All 275 of these triangulations are \emph{bounded} triangulations.  We use
bounded triangulations because they are a stronger ``stress test'':
in practice it has been found that enumerating normal surfaces for
a bounded triangulation is often significantly more difficult than for a
closed triangulation of comparable size \cite{burton12-ws}.
In part this is because the number of admissible vertices of $\qproj$ is
typically much larger in the bounded case \cite{burton10-extreme}.

We begin this section by studying the growth of the number of admissible
vertices and the number of ``dead ends'' that we walk through in the
type tree, where we discover that the number of dead ends is far below the
theoretical bound of $O(4^n)$.
We follow with a direct comparison of running
times for the tree traversal algorithm and the double description
method, where we find that our new algorithm runs slower for smaller
problems but significantly faster for larger and more difficult problems.

\begin{figure}[htb]
    \centering
    \includegraphics[scale=0.6]{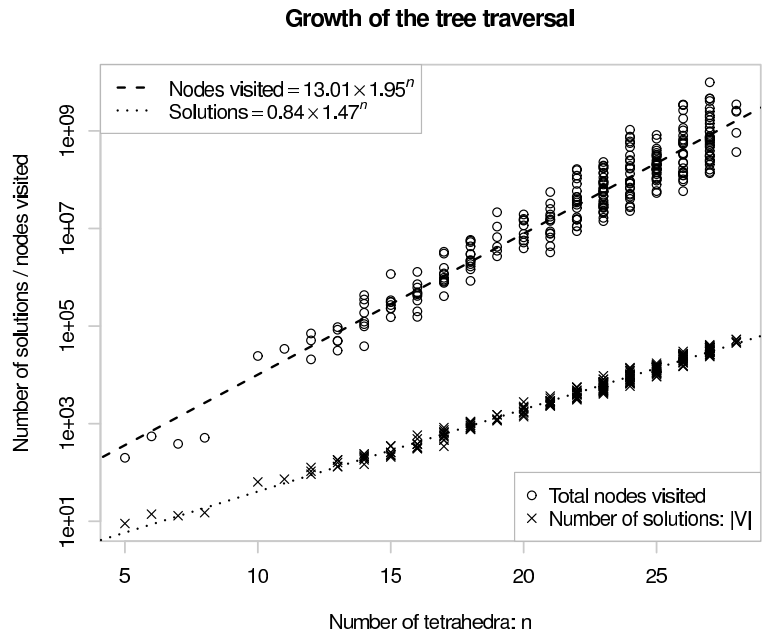}
    \caption{Counting solutions and dead ends for the tree traversal algorithm}
    \label{fig-perf-knots-combo}
\end{figure}

Figure~\ref{fig-perf-knots-combo} measures the ``combinatorial size'' of
the tree traversal: the crosses in the lower part of the graph represent
the final number of solutions,
and the circles in the upper part of the graph count the total
number of nodes that we visit (including solutions as well as dead ends).
In the theoretical analysis of Section~\ref{s-complexity} we estimate
both figures as $O(4^n)$,
but in practice we find that the real counts are significantly smaller.
The number of solutions appears to grow at a
rate of roughly $1.47^n$, and the number of nodes that we visit grows at
a rate of roughly $1.95^n$ (though with some variation).  The corresponding
regression lines are marked on the graph.\footnote{%
    These are standard linear regressions of the form
    $\log y = \alpha n + \beta$, where $y$ is the quantity being
    measured on the vertical axis.}

These small growth rates undoubtedly contribute to the strong performance
of the algorithm, which we discuss shortly.  However,
Figure~\ref{fig-perf-knots-combo} raises another intriguing possibility,
which is that the number of nodes that we visit might be \emph{polynomial
in the output size}.  Indeed, for every triangulation in our test suite,
the number of nodes that we visit is at most $10 |V|^2$ (where $|V|$
represents the number of solutions).  If this were true in general,
both the time and space complexity of the tree traversal algorithm
would become worst-case polynomial in the combined input and output size.
This would be a significant breakthrough for normal surface enumeration.
We return to this speculation in Section~\ref{s-conc}.

We come now to a direct comparison of running times for the tree
traversal algorithm and the double description algorithm.  Both
algorithms are implemented using the topological software package
{\regina} \cite{burton04-regina}.  In particular, the double description
implementation that we use is already built into {\regina}; this represents
the current state of the art for normal surface enumeration, and the details
of the algorithm have been finely tuned over many years
\cite{burton10-dd}.  Both algorithms are implemented in {\cpp}, and
all experiments were run on a single core of a 64-bit 2.3\,GHz AMD Opteron
2356 processor.

\begin{figure}[htb]
    \centering
    \includegraphics[scale=0.6]{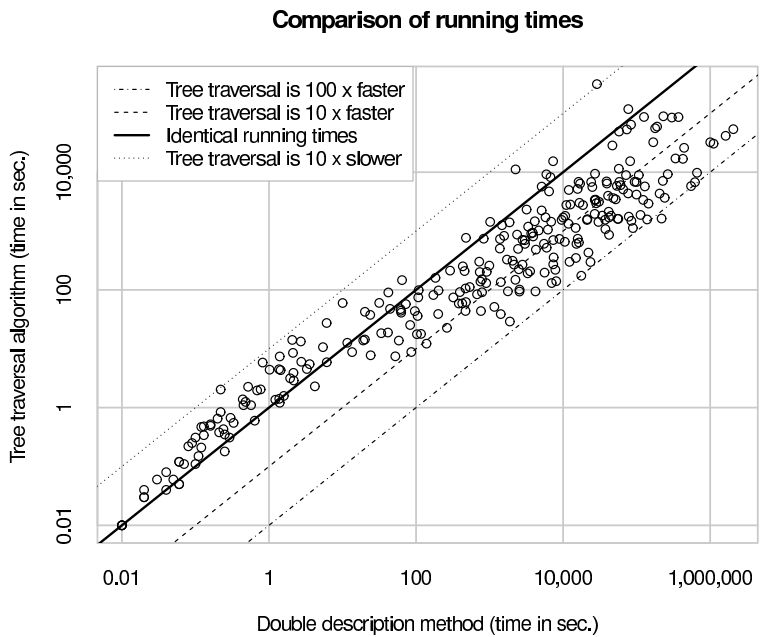}
    \caption{Comparing running times for the old and new algorithms}
    \label{fig-perf-knots}
\end{figure}

Figure~\ref{fig-perf-knots} directly compares the running times for both
algorithms: each point on the graph represents a single triangulation
from the test suite.  Both axes use a logarithmic scale.  The solid
diagonal line indicates where both algorithms have identical running
times, and each dotted line represents an order of magnitude of difference
between them.  Based on this graph we can make the following
observations:
\begin{itemize}
    \item Problems that are difficult for one algorithm are
    difficult for both.  That is, there are no points in the top-left or
    bottom-right regions of the graph.  This is not always the case for
    vertex enumeration problems in general \cite{avis97-howgood-compgeom}.

    \item For smaller problems, the tree traversal algorithm runs slower
    (in the worst case, around 10 times slower).  However, as the
    problems grow more difficult the tree traversal begins to dominate,
    and for running times over 100~seconds the tree traversal
    algorithm is almost always faster (in the best case, around 100
    times faster).
\end{itemize}

\begin{figure}[htb]
    \centering
    \includegraphics[scale=0.6]{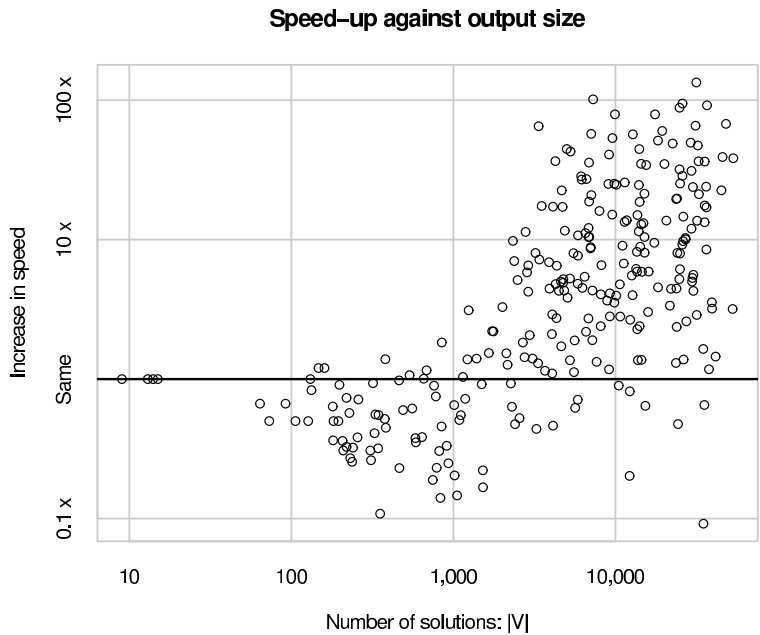}
    \caption{Measuring the performance improvement as the output size grows}
    \label{fig-perf-knots-speedup}
\end{figure}

Figure~\ref{fig-perf-knots-speedup} plots these speed-up factors
explicitly: again each point on the graph represents a single triangulation,
and the vertical axis measures the factor of improvement that we
gain from the tree traversal algorithm.\footnote{Specifically, the
vertical axis measures the double description running time divided
by the tree traversal running time.}
On the horizontal axis we plot the total number of solutions $|V|$,
which is a useful estimate of the difficulty of the problem.
Again, both axes use a logarithmic scale.

Here we see the same behaviour as before, but more visibly: for smaller
problems the tree traversal algorithm runs slower,
and then as the problems grow larger the tree traversal algorithm dominates,
running significantly faster in general for the most difficult cases.

\section{Discussion} \label{s-conc}

For topological decision problems that require lengthy normal surface
enumerations, we conclude that the tree traversal algorithm
is the most promising algorithm amongst those known to date.
Not only does it have small time and space complexities in both
theory and practice (though these complexities remain exponential),
but it also supports parallelisation, progress tracking,
incremental output and early termination.
As described in Section~\ref{s-algm-main}, both parallelisation and
incremental output offer great potential for normal surface theory, yet they
have not been explored in the literature to date.

One aspect of the tree traversal algorithm that has not been discussed
in this paper is the \emph{ordering of tetrahedra}.  By reordering the
tetrahedra we can alter which nodes of the type tree pass the feasibility
test, and thereby reduce the total number of nodes that we visit.
Because this total number of nodes is the source of the exponential running
time, this simple reordering of tetrahedra could have a great effect on the
time complexity (much like reordering the matching equations does in the
double description method \cite{avis97-howgood-compgeom,burton10-dd}).
Further exploration and experimentation in this area could prove beneficial.

We finish by recalling the empirical observations of Section~\ref{s-perf}
that the total number of nodes visited---and therefore the overall
running time---could be a small polynomial in the combined input and
output size.  The data presented in Section~\ref{s-perf} suggest a
quadratic relationship, and similar experimentation on closed triangulations
suggests a cubic relationship.  Even if we examine triangulations with
extremely small output sizes (such as layered lens spaces, where
the number of admissible vertices is linear in $n$ \cite{burton10-dd}),
this small polynomial relationship appears to be preserved (for
layered lens spaces we visit $O(n^3)$ nodes in total).

Although the tree traversal algorithm cannot solve the
vertex enumeration problem for \emph{general} polytopes in polynomial time,
there is much contextual information to be
gained from normal surface theory and the quadrilateral constraints.
Further research into this polynomial time conjecture could
prove fruitful: a counterexample would be informative, and a proof would
be a significant breakthrough
for the complexity of decision algorithms in low-dimensional topology.

%
%

\section*{Acknowledgements}

The authors are grateful to RMIT University for the use of their
high-performance computing facilities.
The first author is supported by the Australian Research Council
under the Discovery Projects funding scheme
(project DP1094516).

%
%

\small
\bibliographystyle{amsplain}
\bibliography{pure}

\newcommand{\noopsort}[1]{}
\providecommand{\bysame}{\leavevmode\hbox to3em{\hrulefill}\thinspace}
\providecommand{\MR}{\relax\ifhmode\unskip\space\fi MR }
\providecommand{\MRhref}[2]{%
  \href{http://www.ams.org/mathscinet-getitem?mr=#1}{#2}
}
\providecommand{\href}[2]{#2}
\begin{thebibliography}{10}

\bibitem{avis00-revised}
David Avis, \emph{A revised implementation of the reverse search vertex
  enumeration algorithm}, Polytopes---Combinatorics and Computation
  (Oberwolfach, 1997), DMV Sem., vol.~29, Birkh{\"a}user, Basel, 2000,
  pp.~177--198.

\bibitem{avis97-howgood-compgeom}
David Avis, David Bremner, and Raimund Seidel, \emph{How good are convex hull
  algorithms?}, Comput. Geom. \textbf{7} (1997), no.~5--6, 265--301.

\bibitem{avis92-pivot}
David Avis and Komei Fukuda, \emph{A pivoting algorithm for convex hulls and
  vertex enumeration of arrangements and polyhedra}, Discrete Comput. Geom.
  \textbf{8} (1992), no.~3, 295--313.

\bibitem{balinski61-backtrack}
M.~L. Balinski, \emph{An algorithm for finding all vertices of convex
  polyhedral sets}, SIAM J. Appl. Math. \textbf{9} (1961), no.~1, 72--88.

\bibitem{bazaraa10-lp}
Mokhtar~S. Bazaraa, John~J. Jarvis, and Hanif~D. Sherali, \emph{Linear
  programming and network flows}, 4th ed., Wiley, Hoboken, NJ, 2010.

\bibitem{birman80-problems}
Joan~S. Birman, \emph{Problem list: Nonsufficiently large 3-manifolds}, Notices
  Amer. Math. Soc. \textbf{27} (1980), no.~4, 349.

\bibitem{bland77-pivoting}
Robert~G. Bland, \emph{New finite pivoting rules for the simplex method}, Math.
  Oper. Res. \textbf{2} (1977), no.~2, 103--107.

\bibitem{burton04-regina}
Benjamin~A. Burton, \emph{Introducing {R}egina, the 3-manifold topology
  software}, Experiment. Math. \textbf{13} (2004), no.~3, 267--272.

\bibitem{burton09-convert}
\bysame, \emph{Converting between quadrilateral and standard solution sets in
  normal surface theory}, Algebr. Geom. Topol. \textbf{9} (2009), no.~4,
  2121--2174.

\bibitem{burton10-complexity}
\bysame, \emph{The complexity of the normal surface solution space}, SCG '10:
  Proceedings of the Twenty-Sixth Annual Symposium on Computational Geometry,
  ACM, 2010, pp.~201--209.

\bibitem{burton10-extreme}
\bysame, \emph{Extreme cases in normal surface enumeration}, In preparation,
  2010.

\bibitem{burton10-dd}
\bysame, \emph{Optimizing the double description method for normal surface
  enumeration}, Math. Comp. \textbf{79} (2010), no.~269, 453--484.

\bibitem{burton11-asymptotic}
\bysame, \emph{Maximal admissible faces and asymptotic bounds for the normal
  surface solution space}, J. Combin. Theory Ser. A \textbf{118} (2011), no.~4,
  1410--1435.

\bibitem{burton12-ws}
Benjamin~A. Burton, J.~Hyam Rubinstein, and Stephan Tillmann, \emph{The
  {W}eber-{S}eifert dodecahedral space is non-{H}aken}, Trans. Amer. Math. Soc.
  \textbf{364} (2012), no.~2, 911--932.

\bibitem{coulson00-snap}
David Coulson, Oliver~A. Goodman, Craig~D. Hodgson, and Walter~D. Neumann,
  \emph{Computing arithmetic invariants of 3-manifolds}, Experiment. Math.
  \textbf{9} (2000), no.~1, 127--152.

\bibitem{fxrays}
Marc Culler and Nathan Dunfield, \emph{{FX}rays: Extremal ray enumeration
  software}, \texttt{http://\allowbreak www.\allowbreak math.\allowbreak
  uic.\allowbreak edu/\allowbreak \~{}t3m/}, 2002--2003.

\bibitem{deghellinck86-newton}
Guy de~Ghellinck and Jean-Philippe Vial, \emph{A polynomial {N}ewton method for
  linear programming}, Algorithmica \textbf{1} (1986), no.~4, 425--453.

\bibitem{dyer83-complexity}
M.~E. Dyer, \emph{The complexity of vertex enumeration methods}, Math. Oper.
  Res. \textbf{8} (1983), no.~3, 381--402.

\bibitem{fukuda97-analysis}
Komei Fukuda, Thomas~M. Liebling, and Fran{\c{c}}ois Margot, \emph{Analysis of
  backtrack algorithms for listing all vertices and all faces of a convex
  polyhedron}, Comput. Geom. \textbf{8} (1997), no.~1, 1--12.

\bibitem{fukuda96-doubledesc}
Komei Fukuda and Alain Prodon, \emph{Double description method revisited},
  Combinatorics and Computer Science (Brest, 1995), Lecture Notes in Comput.
  Sci., vol. 1120, Springer, Berlin, 1996, pp.~91--111.

\bibitem{gmp}
Torbj{\"o}rn Granlund et~al., \emph{The {GNU} multiple precision arithmetic
  library}, \texttt{http://\allowbreak gmplib.\allowbreak org/}, 1991--2010.

\bibitem{haken61-knot}
Wolfgang Haken, \emph{Theorie der {N}ormalfl{\"a}chen}, Acta Math. \textbf{105}
  (1961), 245--375.

\bibitem{haken62-homeomorphism}
\bysame, \emph{{\"U}ber das {H}om{\"o}omorphieproblem der
  3-{M}annigfaltigkeiten. {I}}, Math. Z. \textbf{80} (1962), 89--120.

\bibitem{hass99-knotnp}
Joel Hass, Jeffrey~C. Lagarias, and Nicholas Pippenger, \emph{The computational
  complexity of knot and link problems}, J. Assoc. Comput. Mach. \textbf{46}
  (1999), no.~2, 185--211.

\bibitem{jaco05-lectures-homeomorphism}
William Jaco, \emph{The homeomorphism problem: Classification of 3-manifolds},
  Lecture notes, Available from \texttt{http://\allowbreak www.\allowbreak
  math.\allowbreak okstate.\allowbreak edu/\allowbreak \~{}jaco/\allowbreak
  pekinglectures.\allowbreak htm}, 2005.

\bibitem{jaco02-algorithms-essential}
William Jaco, David Letscher, and J.~Hyam Rubinstein, \emph{Algorithms for
  essential surfaces in 3-manifolds}, Topology and Geometry: Commemorating
  SISTAG, Contemporary Mathematics, no. 314, Amer. Math. Soc., Providence, RI,
  2002, pp.~107--124.

\bibitem{jaco84-haken}
William Jaco and Ulrich Oertel, \emph{An algorithm to decide if a
  {$3$}-manifold is a {H}aken manifold}, Topology \textbf{23} (1984), no.~2,
  195--209.

\bibitem{karmarkar84-new}
N.~Karmarkar, \emph{A new polynomial-time algorithm for linear programming},
  Combinatorica \textbf{4} (1984), no.~4, 373--395.

\bibitem{kleiner08-perelman}
Bruce Kleiner and John Lott, \emph{Notes on {P}erelman's papers}, Geom. Topol.
  \textbf{12} (2008), no.~5, 2587--2855.

\bibitem{kneser29-normal}
Hellmuth Kneser, \emph{Geschlossene {F}l{\"a}chen in dreidimensionalen
  {M}annigfaltigkeiten}, Jahresbericht der Deut. Math. Verein. \textbf{38}
  (1929), 248--260.

\bibitem{matveev03-algms}
Sergei Matveev, \emph{Algorithmic topology and classification of 3-manifolds},
  Algorithms and Computation in Mathematics, no.~9, Springer, Berlin, 2003.

\bibitem{mcmullen70-ubt}
P.~McMullen, \emph{The maximum numbers of faces of a convex polytope},
  Mathematika \textbf{17} (1970), 179--184.

\bibitem{megiddo86-selfdual}
Nimrod Megiddo, \emph{Improved asymptotic analysis of the average number of
  steps performed by the self-dual simplex algorithm}, Math. Programming
  \textbf{35} (1986), no.~2, 140--172.

\bibitem{motzkin53-dd}
T.~S. Motzkin, H.~Raiffa, G.~L. Thompson, and R.~M. Thrall, \emph{The double
  description method}, Contributions to the Theory of Games, Vol. II (H.~W.
  Kuhn and A.~W. Tucker, eds.), Annals of Mathematics Studies, no.~28,
  Princeton University Press, Princeton, NJ, 1953, pp.~51--73.

\bibitem{rubinstein95-3sphere}
J.~Hyam Rubinstein, \emph{An algorithm to recognize the {$3$}-sphere},
  Proceedings of the International Congress of Mathematicians ({Z}{\"u}rich,
  1994), vol.~1, Birkh{\"a}user, 1995, pp.~601--611.

\bibitem{smale83-simplex}
Steve Smale, \emph{On the average number of steps of the simplex method of
  linear programming}, Math. Programming \textbf{27} (1983), no.~3, 241--262.

\bibitem{terzer10-parallel}
Marco Terzer and J{\"o}rg Stelling, \emph{Parallel extreme ray and pathway
  computation}, Parallel Processing and Applied Mathematics, Lecture Notes in
  Comput. Sci., vol. 6068, Springer, Berlin, 2010, pp.~300--309.

\bibitem{tillmann08-finite}
Stephan Tillmann, \emph{Normal surfaces in topologically finite 3-manifolds},
  Enseign. Math. (2) \textbf{54} (2008), 329--380.

\bibitem{tollefson98-quadspace}
Jeffrey~L. Tollefson, \emph{Normal surface {$Q$}-theory}, Pacific J. Math.
  \textbf{183} (1998), no.~2, 359--374.

\bibitem{ziegler95}
G{\"u}nter~M. Ziegler, \emph{Lectures on polytopes}, Graduate Texts in
  Mathematics, no. 152, Springer-Verlag, New York, 1995.

\end{thebibliography}

%
%

\bigskip
\noindent
Benjamin A.~Burton \\
School of Mathematics and Physics, The University of Queensland \\
Brisbane QLD 4072, Australia \\
(bab@maths.uq.edu.au)

\bigskip
\noindent
Melih Ozlen \\
School of Mathematical and Geospatial Sciences, RMIT University \\
GPO Box 2476V, Melbourne VIC 3001, Australia \\
(melih.ozlen@rmit.edu.au)

\end{document}